\documentclass[11pt,reqno]{amsart}

\usepackage{amsmath, amsthm, amsopn, amssymb}
\usepackage{enumerate}
\usepackage{hyperref}

\newtheorem{thm}{Theorem}[section]
\newtheorem{lemma}[thm]{Lemma}
\newtheorem{prop}[thm]{Proposition}
\newtheorem{conj}[thm]{Conjecture}

\theoremstyle{definition}
\newtheorem{claim}[thm]{Claim}

\newcommand{\eps}{\varepsilon}

\newcommand{\ex}{\mathrm{ex}}
\newcommand{\exnk}{\ex(n, K_{k+1})}
\newcommand{\exnr}{\ex(n, K_{r+1})}
\newcommand{\exGnpr}{\ex(G(n,p), K_{r+1})}

\newcommand{\Bin}{\mathrm{Bin}}

\newcommand{\rms}{\binom{r-1}{s}}
\newcommand{\rmms}{\binom{r-2}{s}}
\newcommand{\rpt}{\binom{r+1}{2}}
\newcommand{\rt}{\binom{r}{2}}
\newcommand{\st}{\binom{s}{2}}
\newcommand{\spt}{\binom{s+1}{2}}
\newcommand{\sppt}{\binom{s+2}{2}}

\newcommand{\BB}{\mathcal{B}}
\newcommand{\Ex}{\mathbb{E}}

\newcommand{\HH}{\mathcal{H}}
\newcommand{\KK}{\mathcal{K}}
\renewcommand{\Pr}{\mathbb{P}}

\newcommand{\TT}{\mathcal{T}}
\newcommand{\TTL}{\TT^L}
\newcommand{\TTH}{\TT^H}
\newcommand{\WW}{\mathcal{W}}

\newcommand{\sHH}{{\partial \HH}}

\newcommand{\Iu}{\mathcal{I}_u}

\newcommand{\Ds}{D^*}

\newcommand{\FFH}{\mathcal{F}^H}
\newcommand{\FFI}{\mathcal{F}^I}
\newcommand{\FFR}{\mathcal{F}^R}
\newcommand{\Fs}{\mathcal{F}^*}
\newcommand{\Fst}{\mathcal{F}^*(T)}
\newcommand{\FreH}{\mathcal{F}_n(H)}
\newcommand{\FnmC}{\mathcal{F}_{n,m}(C_{2\ell+1})}
\newcommand{\FnmH}{\mathcal{F}_{n,m}(H)}
\newcommand{\Free}{\mathcal{F}_{n,m}(K_{r+1})}
\newcommand{\Freet}{\mathcal{F}_{n,m}(K_3)}
\newcommand{\Frees}{\mathcal{F}_{n,m}^*(K_{r+1})}
\newcommand{\Freedg}{\mathcal{F}_{n,m}(K_{r+1}; \delta, \gamma)}
\newcommand{\Freesdg}{\mathcal{F}_{n,m}^*(K_{r+1}; \delta, \gamma)}
\newcommand{\Freedgp}{\mathcal{F}_{n,m}(K_{r+1}; \delta, \gamma, \Pi)}
\newcommand{\Freesdgp}{\mathcal{F}_{n,m}^*(K_{r+1}; \delta, \gamma, \Pi)}

\newcommand{\Gnm}{\mathcal{G}_{n,m}}
\newcommand{\Gk}{\mathcal{G}_{n,m}(k)}
\newcommand{\Gr}{\mathcal{G}_{n,m}(r)}
\newcommand{\GP}{\mathcal{G}_m(\Pi)}
\newcommand{\GPp}{\mathcal{G}_m(\Pi')}
\newcommand{\GmP}{\mathcal{G}_{m-1}(\Pi)}
\newcommand{\GmmP}{\mathcal{G}_{m-2}(\Pi)}
\newcommand{\GPe}{\mathcal{G}_m(\Pi,e)}
\newcommand{\GPpf}{\mathcal{G}_m(\Pi',f)}
\newcommand{\GPf}{\mathcal{G}_m(\Pi,f)}

\newcommand{\UP}{\mathcal{U}_m(\Pi)}
\newcommand{\UmP}{\mathcal{U}_{m-1}(\Pi)}
\newcommand{\UmmP}{\mathcal{U}_{m-2}(\Pi)}
\newcommand{\UPe}{\mathcal{U}_m(\Pi,e)}

\newcommand{\Part}{\mathcal{P}_{n,r}}

\newcommand{\Pic}{\Pi^c}
\newcommand{\Pipc}{(\Pi')^c}

\newcommand{\one}{\mathbf{1}}

\newtheorem*{constHH}{Constructing $\HH$}

\renewcommand{\le}{\leqslant}
\renewcommand{\ge}{\geqslant}

\title{The typical structure of sparse $K_{r+1}$-free graphs}

\date{\today}

\author{J\'ozsef Balogh}
\address{Department of Mathematics, University of Illinois, 1409 W. Green Street, Urbana, IL 61801; and Mathematical Institute, University of Szeged, Szeged, Hungary}
\email{jobal@math.uiuc.edu}

\author{Robert Morris} 
\address{IMPA, Estrada Dona Castorina 110, Jardim Bot\^anico, Rio de Janeiro, RJ, Brasil}
\email{rob@impa.br}

\author{Wojciech Samotij}
\address{School of Mathematical Sciences, Tel Aviv University, Tel Aviv 69978, Israel; and Trinity College, Cambridge CB2 1TQ, UK}
\email{samotij@post.tau.ac.il}

\author{Lutz Warnke}
\address{Department of Pure Mathematics and Mathematical Statistics, University of Cambridge, Wilberforce Road, Cambridge CB3 0WB, UK}
\email{L.Warnke@dpmms.cam.ac.uk}

\thanks{Research supported in part by: (JB) Marie Curie Fellowship IIF-327763, NSF CAREER Grant DMS-0745185, UIUC Campus Research Board Grants 11067 and 13039 (Arnold O.\ Beckman Research Award), and OTKA Grant K76099; (RM) CNPq bolsa de Produtividade em Pesquisa; (WS) ERC Advanced Grant DMMCA and Trinity College JRF; (LW) Peterhouse JRF}

\begin{document}

\begin{abstract}
  Two central topics of study in combinatorics are the so-called evolution of random graphs, introduced by the seminal work of Erd{\H{o}}s and R{\'e}nyi, and the family of $H$-free graphs, that is, graphs which do not contain a subgraph isomorphic to a given (usually small) graph $H$. A~widely studied problem that lies at the interface of these two areas is that of determining how the structure of a typical $H$-free graph with $n$ vertices and $m$ edges changes as $m$ grows from $0$ to $\ex(n,H)$. In this paper, we resolve this problem in the case when $H$ is a clique, extending a classical result of Kolaitis, Pr{\"o}mel, and Rothschild. In particular, we prove that for every $r \ge 2$, there is an explicit constant $\theta_r$ such that, letting $m_r = \theta_r n^{2-\frac{2}{r+2}} (\log n)^{1/\left[\rpt-1\right]}$, the following holds for every positive constant $\eps$. If $m \ge (1+\eps) m_r$, then almost all $K_{r+1}$-free $n$-vertex graphs with $m$ edges are $r$-partite, whereas if $n \ll m \le (1-\eps)m_r$, then almost all of them are not $r$-partite. 
\end{abstract}

\maketitle

\section{Introduction}

\subsection{Background and motivation}

Given integers $n$ and $m$, let $G_{n,m}$ be the uniformly chosen random element of the family $\Gnm$ of all graphs on a fixed vertex set of size $n$ that have precisely $m$ edges. The study of the evolvement of typical properties of $G_{n,m}$ when we let $m$ gradually increase from $0$ to $\binom{n}{2}$, known as the \emph{evolution of random graphs}, which was initiated in the seminal work of Erd{\H{o}}s and R{\'e}nyi~\cite{ErRe60}, is a central topic in graph theory. The behavior of many parameters and properties during the evolution of $G_{n,m}$, such as connectivity, containment of small subgraphs, chromatic number, to name a few, is now fairly well understood~\cite{Bo, JaLuRu}. A natural problem is to consider such evolution when we restrict our attention to a certain subclass of graphs, i.e., when $G_{n,m}$ is a random element of some proper subfamily of $\Gnm$.

In this paper, we consider the class of graphs that do not contain a clique of a given fixed order. The study of \emph{$H$-free graphs}, i.e., graphs which do not contain a subgraph isomorphic to a given fixed graph $H$, is one of the cornerstones of extremal graph theory. The classical theorem of Tur{\'a}n~\cite{Tu41} states that for every $r \ge 2$, the largest number of edges in a $K_{r+1}$-free graph on $n$ vertices, denoted $\exnr$, is equal to the number of edges in the balanced complete $r$-partite graph $T_r(n)$, that is
\[
\exnr = e(T_r(n)) = \left(1 - \frac{1}{r}\right)\binom{n}{2} + O(n).
\]
Moreover, it identifies $T_r(n)$ as the unique \emph{extremal} graph, i.e., the unique $K_{r+1}$-free $n$-vertex graph with $\exnr$ edges. A famous result of Kolaitis, Pr{\"o}mel, and Rothschild~\cite{KoPrRo87} determines the \emph{typical} structure of $K_{r+1}$-free graphs. It states that for every $r \ge 2$, almost all $K_{r+1}$-free graphs are $r$-partite ($r$-colorable); in the case $r = 2$, this was proved earlier by Erd{\H{o}}s, Kleitman, and Rothschild~\cite{ErKlRo}.

In view of the above, one naturally arrives at the following question, first considered by Pr{\"o}mel and Steger~\cite{PrSt96} more than fifteen years ago. Let $\Free$ denote the family of all $K_{r+1}$-free graphs on a fixed set of $n$ vertices (for concreteness, we let it be the set $\{1, \ldots, n\}$) that have exactly $m$ edges. For which $m$ are almost all graphs in $\Free$ $r$-partite? This is trivially true for very small $m$, as then almost all graphs in $\Gnm$ are both $K_{r+1}$-free and $r$-colorable, and when $m = \exnr$, by Tur{\'a}n's theorem. By the Kolaitis--Pr{\"o}mel--Rothschild theorem, it must also be true for at least one value of $m$ that is close to $\frac{1}{2} \cdot \exnr$, since almost all $r$-colorable graphs have roughly this many edges. On the other hand, it is not very hard to see that this statement is not true for $m$ in some intermediate range. For example, if $n \ll m \ll n^{4/3}$, then almost all graphs in $\Gnm$ are both $K_4$-free and not $r$-colorable for every fixed $r$.

In the case of triangle-free graphs, it turns out that as $m$ grows from $0$ to $\ex(n,K_3)$, there are two critical points at which almost all graphs in $\Freet$ first stop being and then become bipartite, as proved by Osthus, Pr{\"o}mel, and Taraz~\cite{OsPrTa03}, who improved an earlier result of Pr{\"o}mel and Steger~\cite{PrSt96} (see also Steger~\cite{St05} for a slightly weaker result). More precisely, the following was shown in~\cite{OsPrTa03}. Let $\eps$ be an arbitrary positive constant and let
\begin{equation}
  \label{eq:m2}
  m_2 = m_2(n) = \frac{\sqrt{3}}{4} n^{3/2} \sqrt{\log n}.
\end{equation}
First, if $m \ll n$, then almost all graphs in $\Freet$ are bipartite. Second, if $n/2 \le m \le (1-\eps)m_2$, then almost all these graphs are not bipartite. Third, if $m \ge (1+\eps)m_2$, then again almost all of them are bipartite. A corresponding result for $r = 4$ was obtained in the unpublished master's thesis of the fourth author~\cite{Wa09}.

\subsection{Main result}

In this paper, we generalize the above result to all $r$. To this end, for each $r \ge 2$, define
\begin{equation}
  \label{eq:thetar}
  \theta_r = \frac{r-1}{2r} \cdot \left[ r \cdot \left(\frac{2r+2}{r+2}\right)^{\frac{1}{r-1}} \right]^{\frac{2}{r+2}}
\end{equation}
and
\begin{equation}
  \label{eq:mr}
  m_r = m_r(n) = \theta_r n^{2-\frac{2}{r+2}} (\log n)^{\frac{1}{\rpt-1}}.
\end{equation}
Here and throughout the paper, $\log$ denotes the natural logarithm. Note that the definitions of $m_2$ given by~\eqref{eq:m2} and by~\eqref{eq:mr} coincide. Our main result is the following.

\begin{thm}
  \label{thm:main}
  For every $r \ge 3$, there exists a $d_r = d_r(n) = \Theta(n)$ such that the following holds for every $\eps > 0$. If $F_{n,m}$ is the uniformly chosen random element of $\Free$, then
  \[
  \lim_{n \to \infty} \Pr[\text{$F_{n,m}$ is $r$-partite}] =
  \begin{cases}
    1 & m \le (1-\eps)d_r, \\
    0 & (1+\eps)d_r \le m \le (1-\eps)m_r, \\
    1 & m \ge (1+\eps)m_r.
  \end{cases}
  \]
\end{thm}

The existence of the first threshold and the function $d_r$ in Theorem~\ref{thm:main} follows directly from the fact that for every $r \ge 3$, the property of being $r$-colorable has a sharp threshold in $\Gnm$, as proved by Achlioptas and Friedgut~\cite{AcFr99}, see also~\cite{Fr99}. Indeed, if $m \ll n^{2 - 2/r}$, then almost all graphs in $\Gnm$ are $K_{r+1}$-free and therefore almost every graph in $\Free$ is $r$-partite if and only if almost every graph in $\Gnm$ is $r$-partite. Moreover, one immediately sees that the threshold function for the property of being $r$-colorable, which we denote by $d_r$, satisfies $d_r(n) = \Theta(n)$; for more precise estimates, with which we will not be concerned in this paper, we refer the reader to~\cite{AcNa, CoVi13}. Thus, the main business of this paper will be establishing the existence of the second threshold at $m_r$. Finally, we would like to point out that our arguments for $m \gg n$, which really is the main case of interest for us, are also valid in the case $r = 2$. The only reason why in the statement of Theorem~\ref{thm:main}, we assume that $r \ge 3$ is that the property of being bipartite does not have a sharp threshold in $\Gnm$ and therefore in the case $r = 2$, there is no double sharp threshold phenomenon.

\subsection{Approximate version}

A closely related problem is that of determining for which $m$ almost every graph in $\Free$ is \emph{almost} $r$-partite, i.e., becomes $r$-partite after deleting from it some small fraction of the edges. Fifteen years ago, this problem was first considered by {\L}uczak~\cite{Lu00}, who proved that when $m \gg n^{3/2}$, then almost every graph in $\Freet$ can be made bipartite by deleting from it some $o(m)$ edges. Furthermore, {\L}uczak showed that the so-called K{\L}R conjecture~\cite{KoLuRo97} implies an analogous statement for arbitrary $r \ge 2$, Theorem~\ref{thm:approx-struct} below. This conjecture was only very recently verified by the first three authors~\cite{BaMoSa12}, and by Saxton and Thomason~\cite{SaTh}; see also~\cite{CoGoSaSc}. The following result was established by the first three authors in~\cite{BaMoSa12} (the case $r = 2$ was proved much earlier in~\cite{Lu00}). It may also be derived from the results of~\cite{SaTh}.

\begin{thm}[{\cite{BaMoSa12, SaTh}}]
  \label{thm:approx-struct}
  For every $r \ge 2$ and every $\delta > 0$, there exists a constant~$C$ such that if $m \ge Cn^{2-2/(r+2)}$, then almost every graph in $\Free$ can be made $r$-partite by removing from it at most $\delta m$ edges.
\end{thm}

As mentioned above, Theorem~\ref{thm:approx-struct} was derived from the (then unproven) K{\L}R conjecture in~\cite{Lu00}. We remark here that in fact the proof of Theorem~\ref{thm:approx-struct} in \cite{BaMoSa12} yields that the proportion of graphs in $\Free$ that cannot be made $r$-partite by removing from them some $\delta m$ edges is at most $(1-\eps)^m$ for some positive constant $\eps$ that depends solely on $r$ and $\delta$.

\subsection{Related work}

The main result of this paper, Theorem~\ref{thm:main}, may be also viewed in the context of the recent developments of `sparse random analogues' of classical results in extremal combinatorics such as the aforementioned theorem of Tur{\'a}n. Here, we just give a very brief summary of these developments. For more information, we refer the interested reader to the survey of R{\"o}dl and Schacht~\cite{RoSc13}. A long line of research, initiated in~\cite{BaSiSp90, FrRo86, KoLuRo96, RoRu95, RoRu97}, has recently culminated in breakthroughs of Conlon and Gowers~\cite{CoGo} and Schacht~\cite{Sc} (see also \cite{BaMoSa12, CoGoSaSc, FrRoSc10, Sa14, SaTh}), which developed a general theory for approaching such problems. In particular, these results imply that, asymptotically almost surely (a.a.s.), i.e., with probability tending to one as $n \to \infty$, if $p \gg n^{-2/(r+2)}$, then the number $\exGnpr$ of edges in a largest $K_{r+1}$-free subgraph of the binomial random graph $G(n,p)$ satisfies
\[
\exGnpr = \left(1 - \frac{1}{r}\right)\binom{n}{2}p + o(n^2p).
\]
This statement is usually referred to as the sparse random analogue of Tur{\'a}n's theorem. Moreover, a.a.s.\ any $K_{r+1}$-free subgraph of $G(n,p)$ with  $\exGnpr - o(n^2p)$ edges may be made $r$-partite by removing from it some $o(n^2p)$ edges. This is usually referred to as the sparse random analogue of the Erd{\H{o}}s--Simonovits stability theorem~\cite{ErSi66, Si68}.

In fact, these random analogues of the theorems of Tur{\'a}n and Erd{\H{o}}s and Simonovits are much more closely related to Theorem~\ref{thm:approx-struct} rather than Theorem~\ref{thm:main}. A question somewhat closer in spirit to the latter would be deciding for which functions $p = p(n)$ is, a.a.s., the largest $K_{r+1}$-free subgraph of $G(n,p)$ exactly $r$-partite. Such a statement may be viewed as an exact sparse random analogue of Tur{\'a}n's theorem. This problem also has a fairly long history.  It was first considered by Babai, Simonovits, and Spencer~\cite{BaSiSp90}, who proved that the condition $p > 1/2$ is sufficient in the case $r = 2$. Much later, Brightwell, Panagiotou, and Steger~\cite{BrPaSt12} showed that for every $r$, the exact random analogue of Tur{\'a}n's theorem holds when $p(n) \ge n^{-c}$ for some constant $c$ that depends on $r$. Very recently, DeMarco and Kahn~\cite{DMKa} showed that in the case $r = 2$ it is enough to assume that $p(n) \ge C \sqrt{\log n / n}$, where $C$ is some positive constant; this is best possible up to the value of $C$ since, as pointed out in~\cite{BrPaSt12}, the statement is false when $p(n) \le \frac{1}{10}\sqrt{\log n / n}$. Note that the threshold $p(n) = \sqrt{\log n / n}$ for the above property (in the case $r = 2$) in $G(n,p)$ coincides with the threshold $m_2(n)$ from Theorem~\ref{thm:main}; nevertheless, the difficulties encountered in the two problems are quite different, and the proof in~\cite{DMKa} has (not surprisingly) little in common with that given here. Even more recently, DeMarco and Kahn~\cite{DMKa-Kr} showed that for every $r \ge 2$, the exact random analogue of Tur{\'a}n's theorem for $K_{r+1}$ holds under the assumption that $p(n) \ge C m_r(n) / n^2$, where $C$ is a sufficiently large constant and $m_r(n)$ is defined in~\eqref{eq:mr}, see~\cite[Conjecture~1.2]{DMKa}.

Finally, we remark that the family $\FreH$ of $n$-vertex $H$-free graphs has been extensively studied for general graphs $H$. Extending~\cite{KoPrRo87}, Pr{\"o}mel and Steger~\cite{PrSt92} proved that if $H$ contains an edge whose removal reduces the chromatic number of $H$ (such graphs are called \emph{edge-color-critical}), then almost all graphs in $\FreH$ are $(\chi(H)-1)$-partite. Generalizing \cite{KoPrRo87, PrSt92} even further, Hundack, Pr\"omel, and Steger~\cite{HuPrSt93} proved that if $H$ contains a \emph{color-critical vertex} (one whose removal reduces the chromatic number), then almost every graph in $\FreH$ admits a partition of its vertex set into $\chi(H)-1$ parts, each of which induces a subgraph whose maximum degree is bounded by an explicit constant $d_H$ (in particular, $d_H = 0$ if $H$ is edge-color-critical). It would be interesting to generalize these results in the same way that Theorem~\ref{thm:main} generalizes the Kolaitis--Pr{\"o}mel--Rothschild theorem. We expect the following statement to be true.

\begin{conj}
  \label{sec:conj-main}
  For every strictly $2$-balanced, edge-color-critical graph $H$, there exists a constant $C$ such that the following holds. If
  \[
  m \ge Cn^{2 - 1/m_2(H)} (\log n)^{1/(e(H)-1)},
  \]
  then almost all graphs in $\FnmH$ are $(\chi(H)-1)$-partite.
\end{conj}

Let us remark here that a statement that is even stronger than Conjecture~\ref{sec:conj-main} was proved by Osthus, Pr{\"o}mel, and Taraz~\cite{OsPrTa03} in the case when $H$ is a cycle of odd length. More precisely, the following was shown in~\cite{OsPrTa03}. Let $\ell$ be an integer, let $\eps$ be an arbitrary positive constant, and let
\[
t_\ell = t_\ell(n) = \left(\frac{\ell}{\ell-1} \cdot \left(\frac{n}{2}\right)^\ell \log n \right)^{\frac{1}{\ell-1}}.
\]
If $n/2 \le m \le (1-\eps)t_\ell$, then almost all graphs in $\FnmC$ are not bipartite and if $m \ge (1+\eps)t_\ell$, then almost all of them are bipartite.

Many (almost) sharp results describing the structure of almost all graphs in $\FreH$ for a general graph $H$ were proved in a series of papers by Balogh, Bollob{\'a}s, and Simonovits~\cite{BaBoSi04, BaBoSi09, BaBoSi11}. Although it is not explicitly stated there, one can read out from their proofs that almost all graphs in $\Free$ are $r$-partite when $m = \Omega(n^2)$.

Precise structural descriptions of the families $\FreH$ when $H$ is a hypergraph are harder to obtain, and such results have so far been proved only for a few specific $3$-uniform hypergraphs~\cite{BaMu11, BaMu12, PeSc09}. Finally, we remark that the typical structure of graphs with a forbidden \emph{induced} subgraph has also been considered in the literature, see~\cite{AlBaBoMo11, BaBu11}.

\subsection{Outline of the paper}

The remainder of this paper is organised as follows. In Section~\ref{sec:outline-proof}, we give a fairly detailed outline of the strategy for proving Theorem~\ref{thm:main}. In Sections~\ref{sec:preliminaries} and~\ref{sec:r-col-graphs}, we collect some auxiliary results needed for the proof. In Section~\ref{sec:0-statement}, we establish the $0$-statement in Theorem~\ref{thm:main}. Starting with Section~\ref{sec:1-statement}, we turn to proving the second $1$-statement in Theorem~\ref{thm:main}. Our argument is rather involved and we subdivide it into two different cases, addressed in Sections~\ref{sec:sparse-case} and~\ref{sec:dense-case}, respectively.

\subsection{Notation}

For the sake of brevity, given an integer $n$, we will abbreviate $\{1, \ldots, n\}$ by $[n]$. For concreteness, we assume that $[n]$ is the common vertex set of all of the $n$-vertex graphs we consider in this paper. Let $\Gr$ be the family of all graphs in $\Gnm$, i.e., graphs on the vertex set $[n]$ with precisely $m$ edges, that are $r$-colorable. Let $\Part$ be the family of all $r$-colorings of $[n]$, that is, all partitions of $[n]$ into $r$ parts. For the sake of brevity, we shall often identify a partition $\Pi \in \Part$ with the complete $r$-partite graph on the vertex set $[n]$ whose color classes are the $r$ parts of $\Pi$. In particular, if $G$ is a graph on the vertex set $[n]$, then $G \subseteq \Pi$ means that $G$ is a subgraph of the complete $r$-partite graph $\Pi$ or, in other words, the partition $\Pi$ is a proper coloring of $G$. Exploiting this convention, we will also write $\Pic$ to denote the complement of the graph $\Pi$, that is, the union of $r$ complete graphs whose vertex sets are the color classes of $\Pi$.

Otherwise, we use fairly standard conventions. In particular, given a graph $G$, one of its vertices $v$, and a set $A \subseteq V(G)$, we denote the number of edges in $G$, the degree of $v$ in $G$, and the number of neighbors of $v$ in the set $A$ by $e(G)$, $\deg_G(v)$, and $\deg_G(v,A)$, respectively. For a graph $G$ and a set $A \subseteq V(G)$, we shall write $G - A$ to denote the subgraph of $G$ induced by the set $V(G) \setminus A$. Perhaps less obviously, $K_{r+1}^-$ denotes the graph obtained from the complete graph on $r+1$ vertices by removing a single edge, which we refer to as the \emph{missing edge}. For the sake of clarity of presentation, we will often assume that all large numbers are integers and use floor and ceiling symbols only when we feel that not writing them explicitly might be confusing. Our asymptotic notation is also standard; in particular, we write $f(n) \ll g(n)$ to denote the fact that $f(n)/g(n) \to 0$ as $n \to \infty$. Finally, we adapt the following notational convention. A subscript of the form X.Y refers to Claim/Lemma/Proposition/Theorem X.Y. For example, we write $\delta_{\ref{prop:approx-struct}}(\cdot)$ to denote the function implicitly defined in the statement of Proposition~\ref{prop:approx-struct}.

\section{Outline of the proof}

\label{sec:outline-proof}

In this section, we outline the proof of our main result, Theorem~\ref{thm:main}. The case $m = O(n)$ was already extensively discussed in the paragraph following the statement of Theorem~\ref{thm:main}, so in the remainder of the paper, we will assume that $m \gg n$. This leaves us with proving the following two statements, which we will do for every $r \ge 2$. First, we will show that if $m \le (1-\eps)m_r$, then almost all graphs in $\Free$ are not $r$-partite. Second, we will show that if $m \ge (1+\eps)m_r$, then almost all graphs are $r$-partite. For the sake of brevity, we will refer to these two assertions as the \emph{$0$-statement} and the \emph{$1$-statement}, respectively. We start by giving a heuristic argument that suggests that the function $m_r$ defined in~\eqref{eq:mr} is indeed the sharp threshold.

\subsection{Why is $m_r$ the sharp threshold?}

\label{sec:mr-threshold}

For the sake of clarity of presentation, we shall introduce another parameter that is related to the threshold function $m_r$. We let $p_r = p_r(n)$ be the number satisfying
\begin{equation}
  \label{eq:pr-def}
  \left(\frac{n}{r}\right)^{r-1} p_r^{\rpt-1} = \left(2-\frac{2}{r+2}\right) \log n
\end{equation}
and note that
\begin{equation}
  \label{eq:pr-mr}
  m_r = \left(1 - \frac{1}{r}\right)\frac{n^2}{2} \cdot p_r \approx \exnr \cdot p_r.
\end{equation}
Although~\eqref{eq:pr-def} might seem like a strange way of defining $p_r$, given that we also have~\eqref{eq:pr-mr}, we shall see that considering~\eqref{eq:pr-def} and~\eqref{eq:pr-mr} in the above order is the natural way of arriving at the threshold $m_r$. Recall that we are aiming to count all non-$r$-partite graphs in $\Free$. Consider a \emph{fixed} $r$-partition $\Pi = \{V_1, \ldots, V_r\}$ that is \emph{balanced}, that is, such that $|V_i| \approx \frac{n}{r}$ for all $i$. (As we shall show in Section~\ref{sec:r-col-graphs}, if $m \gg n$, then almost all $r$-colorable graphs with $m$ edges admit only balanced proper $r$-colorings.) Note that the assumption that $\Pi$ is balanced implies that for every $i$,
\begin{equation}
  \label{eq:Pi-balanced-approx}
  e(\Pi) \approx \left(1-\frac{1}{r}\right)\frac{n^2}{2} \quad \text{and} \quad \prod_{j \neq i} |V_j| \approx \left(\frac{n}{r}\right)^{r-1}.
\end{equation}
Let us try to count the graphs in $\Free$ that are not $r$-partite, but for which $\Pi$ is an \emph{almost} proper $r$-coloring, i.e., graphs with exactly one monochromatic edge in the coloring $\Pi$. The presence of a monochromatic edge $\{v, w\} \subseteq V_i$ in some $G \in \Free$ implies that $G$ has to \emph{avoid} $\prod_{j \neq i} |V_j|$ copies of $K_{r+1}^-$ in $\Pi$, where $\{v, w\}$ is the missing edge. More precisely, no such $G$ can contain all $\rpt-1$ edges in any such copy of $K_{r+1}^-$. The proportion of subgraphs of $\Pi$ with $m-1$ edges that avoid a single copy of $K_{r+1}^-$ is about $1-(\frac{m}{e(\Pi)})^{\rpt-1}$. Therefore, if not containing different copies of $K_{r+1}^-$ were independent events in the space of all subgraphs of $\Pi$, then, by \eqref{eq:pr-def}, \eqref{eq:pr-mr}, and \eqref{eq:Pi-balanced-approx}, if $m \approx m_r$, then the proportion $P$ of subgraphs avoiding all copies supported on the edge $\{v, w\}$ would satisfy, roughly,
\[
P \approx \left(1 - \left(\frac{m}{e(\Pi)}\right)^{\rpt - 1}\right)^{\prod_{j \neq i} |V_j|} \approx \exp\left( - \left(\frac{n}{r}\right)^{r-1} p_r^{\rpt-1} \right) = n^{-2+\frac{2}{r+2}} \approx \frac{1}{m}.
\]
This hints that $m \approx m_r$ is a `critical point' as the number of graphs in $\Free$ that have precisely one monochromatic edge in the coloring $\Pi$ is
\[
P \binom{e(\Pic)}{1} \binom{e(\Pi)}{m-1},
\]
which is of the same order as $\binom{e(\Pi)}{m}$, the number of graphs in $\Gnm$ which are properly colored by $\Pi$, exactly when $P = \Theta(\frac{1}{m})$.

\subsection{Sketch of the proof of the $0$-statement}

\label{sec:sketch-proof-0}

Here, we assume that $(1+\eps)d_r \le m \le (1-\eps)m_r$. As we have already mentioned in the introduction, if $m \ll n^{2-2/r}$, then almost all graphs in $\Gnm$ are $K_{r+1}$-free and therefore the fact that almost every graph in $\Free$ is not $r$-partite follows from the fact that $d_r$ is the sharp threshold for the property of being $r$-colorable in $\Gnm$. The existence of such a sharp threshold (for $r \ge 3$) was proved by Achlioptas and Friedgut~\cite{AcFr99}, see also~\cite{Fr99}. Moreover, a fairly straightforward counting argument employing the Hypergeometric FKG Inequality (Lemma~\ref{lemma:HFKG}) shows that if $m \ll n^{2 - 2/(r+2)}$, then the number of graphs in $\Free$ is far greater than $|\Gr|$, the number of $r$-colorable graphs in $\Gnm$, see Section~\ref{sec:very-sparse}. Therefore, we focus our attention on the case when $m = \Omega(n^{2-2/(r+2)})$.

As already suggested in Section~\ref{sec:mr-threshold}, the main idea is to count graphs in $\Free$ that are $r$-colorable except for one edge. This suffices for our purposes, as it turns out that the number of such graphs is already asymptotically greater than the number of graphs in $\Free$ that are $r$-colorable. To this end, we first show, in Lemma~\ref{lemma:GPe-free}, that for a fixed $r$-coloring $\Pi$ of $[n]$ that is balanced, that is, each of its color classes has size $n/r + o(n)$, the number of graphs $G \in \Free$ such that $e(G \cap \Pic) = 1$ is asymptotically greater than the number of graphs in $\Free$ which are properly colored by $\Pi$. Recall from the discussion above that a graph $G$ with $e(G \cap \Pic) = 1$ is $K_{r+1}$-free precisely when none of the about $(n/r)^{r-1}$ copies of $K_{r+1}^-$ contained in $\Pi$, where the unique edge of $G \cap \Pic$ is the missing edge, is completely contained in $G \cap \Pi$. Using this simple observation, we obtain a lower bound on the number of such $G$ using the Hypergeometric FKG Inequality (Lemma~\ref{lemma:HFKG}). Our bound implies that for a fixed balanced $\Pi$, the number of graphs $G \in \Free$ with exactly one edge in $\Pic$ is much larger than $\binom{e(\Pi)}{m}$, the number of graphs which are properly colored by $\Pi$. Second, in Lemma~\ref{lemma:GPe-unique}, we show that almost every $G \in \Free$ that admits an $r$-coloring $\Pi$ satisfying $e(G \cap \Pic) = 1$ admits a unique such $\Pi$. Consequently, the number of such $G$ is almost as large as the sum over all balanced $r$-colorings $\Pi$ of the lower bounds obtained earlier, and this is much larger than $|\Gr|$, the number of $r$-colorable graphs in $\Gnm$.

\subsection{Sketch of the proof of the $1$-statement}

\label{sec:sketch-proof-1}

Here, we assume that $m \ge (1+\eps)m_r$ for some positive constant $\eps$. Since in particular $m \gg n^{-\frac{2}{r+2}}$, then Theorem~\ref{thm:approx-struct} implies that almost every graph in $\Free$ admits an $r$-coloring of $[n]$ such that:
\begin{enumerate}[(i)]
\item
  there are only $o(m)$ monochromatic edges,
\item
  each color class has size $n/r + o(n)$,
\item
  if a vertex $v$ is colored $i$, then $v$ has at least as many neighbors in every color $j$ as in color $i$.
\end{enumerate}
Therefore, it suffices to consider only $K_{r+1}$-free graphs that admit such a coloring.

As was proved in~\cite{PrSt95}, almost every graph in $\Gr$ admits a unique $r$-coloring. Moreover, the number of pairs $(G,\Pi)$, where $G \in \Gr$ and $\Pi$ is a proper $r$-coloring of $G$ is asymptotic to $|\Gr|$, see Theorem~\ref{thm:Gr}. Therefore, it suffices to prove that for a \emph{fixed} $r$-coloring $\Pi$ satisfying (ii) above, the number of $G \in \Free$ that satisfy (i) and (iii) for this fixed coloring $\Pi$ is asymptotically equal to $\binom{e(\Pi)}{m}$, the number of graphs in $\Gnm$ that are properly colored by $\Pi$, see Theorem~\ref{thm:1-statement}.

From now on, we fix some $\Pi$ satisfying (ii) and count graphs $G \in \Free$ that satisfy (i) and (iii) but are \emph{not} properly colored by $\Pi$. We denote the family of all such graphs by $\Fs$. The methods of enumerating these graphs will vary with $m$ and the distribution of the monochromatic edges of $G$, that is, the edges of $G \cap \Pic$. For technical reasons, we require separate arguments to handle the cases $m \le \exnr - \xi n^2$ and $m > \exnr - \xi n^2$, where $\xi$ is some fixed positive constant, which we term the \emph{sparse case} and the \emph{dense case}, respectively. The argument used for the (much easier) dense case is somewhat ad hoc and we will not dwell on it here. Instead, we refer the interested reader directly to (the self-contained) Section~\ref{sec:dense-case}. The main business of this paper is handling the sparse case, and hence from now on we assume that $m \le \exnr - \xi n^2$.

Recall that $\Pi$ is a fixed $r$-coloring of $[n]$ satisfying (ii). We use two different methods of enumerating graphs $G \in \Fs$, that is, graphs in $\Free$ that satisfy (i) and (iii) but are not properly colored by $\Pi$, depending on whether or not most edges of $G \cap \Pic$ are incident to vertices whose degree in $G \cap \Pic$ is somewhat high. More precisely, we partition the family $\TT$ of all possible graphs $G \cap \Pic$, where $G$ ranges over $\Fs$, into two classes, denoted $\TTL$ (here $L$ stands for \emph{low degree}) and $\TTH$ (here $H$ stands for \emph{high degree}), according to the proportion of edges that are incident to vertices whose degree exceeds $\beta m / n$, where $\beta$ is a small positive constant, see Section~\ref{sec:setup-sparse}. We then separately enumerate graphs $G$ such that $G \cap \Pic \in \TTL$ and those satisfying $G \cap \Pic \in \TTH$. We term these two parts of the argument the \emph{low degree case} (Section~\ref{sec:sparse-low-degree-case}) and the \emph{high degree case} (Sections~\ref{sec:sparse-high-degree-case}--\ref{sec:irregular-case}), respectively.

In the (easier) low degree case, for each $T \in \TTL$, we give an upper bound on the number of graphs $G \in \Fs$ such that $G \cap \Pic = T$. Our upper bound is a function of the number of edges in a canonically chosen subgraph $U(T)$ of $T$, which we define in Section~\ref{sec:setup-sparse}. The bound is proved in Lemma~\ref{lemma:Janson}, which is the core of the argument in the low degree case. In Section~\ref{sec:counting-graphs}, we separately enumerate all $T \in \TTL$ with a certain value of $e(U(T))$. This is fairly straightforward. The proof of Lemma~\ref{lemma:Janson}, which bounds the number of $G$ with $G \cap \Pic = T \in \TT$ in terms of $e(U(T))$, is a somewhat involved application of the Hypergeometric Janson Inequality (Lemma~\ref{lemma:HJI}). Let us briefly describe the main idea. The presence of an edge $\{v ,w\}$ of $T$ in $G \cap \Pic$ and the fact that $G$ is $K_{r+1}$-free imply that $G \cap \Pi$ avoids each of the (roughly) $(n/r)^{r-1}$ copies of $K_{r+1}^-$ in $\Pi$, where $v$ and $w$ are the endpoints of the missing edge of $K_{r+1}^-$. That is, at least one edge from each such copy of $K_{r+1}^-$ does not belong to $G \cap \Pi$. Whereas it is quite easy to estimate the number of subgraphs of $\Pi$ with $m - e(T)$ edges that avoid a \emph{single} copy of $K_{r+1}^-$ (there are about $\left(1-(\frac{m}{e(\Pi)})^{\rpt-1}\right) \cdot \binom{e(\Pi)}{m-e(T)}$ of them) bounding the number of subgraphs that avoid \emph{all} such copies of $K_{r+1}^-$ for \emph{all} edges $\{v, w\}$ of $T$ simultaneously requires very careful computation. The main difficulty lies in controlling the correlation between the families of subgraphs of $\Pi$ that avoid two different copies of $K_{r+1}^-$.

In the high degree case, where we enumerate the graphs $G \in \Fs$ such that $G \cap \Pic \in \TTH$, we focus our attention on vertices of high degree, that is, vertices whose degree in $G \cap \Pic$ exceeds $\beta m / n$. We count such graphs by describing and analyzing a procedure that constructs all of them in two stages. This procedure first selects one color class, $V_i$, chooses which of its vertices will have high degree and then picks their neighbors, in all color classes. Next, it chooses all the remaining edges of $G$. In the analysis, we bound the number of choices that this procedure can make, which translates into a bound on the number of graphs in $\Fs$ that fall into the high degree case. Let us briefly describe how we obtain this bound. Suppose that we want to construct a graph $G \in \Fs$, where some $v \in V_i$ has at least $\beta m / n$ neighbors in $V_i$. By (iii) above, we must guarantee that $\deg(v,V_j) \ge \beta m / n$ for all $j$. Hence, no matter how we choose the neighborhoods of $v$ in $V_1, \ldots, V_r$, there will be a collection of at least $(\beta m / n)^r$ \emph{forbidden copies} of $K_r$ in $\Pi$, none of which can be fully contained in $G$. For a typical choice of neighborhoods of some canonically chosen set of high degree vertices in $V_i$ (in the first stage of the procedure), these forbidden copies of $K_r$ are uniformly distributed and hence, using the Hypergeometric Janson Inequality (Lemma~\ref{lemma:HJI}), we can obtain a strong upper bound on the number of choices of a subgraph of $\Pi$ that avoids them (in the second stage of our procedure). We will refer to this possibility as the \emph{regular case}. On the other hand, using Lemma~\ref{lemma:d-sets}, we show that there are only very few choices of the neighborhoods of these high degree vertices in $V_i$ (in the first stage) for which the distribution of the forbidden copies of $K_r$ is not sufficiently uniform to yield a strong bound on the number of choices of the remaining edges (in the second stage), as in the regular case. We will refer to this possibility as the \emph{irregular case}. The proportion of graphs that fall into either the regular or the irregular case is exponentially small in $m/n$.

\section{Preliminaries}

\label{sec:preliminaries}

\subsection{Tools}

In this section, we collect several auxiliary results that will be later used in the proof of Theorem~\ref{thm:main}. We begin with one of our main tools, a version of the Janson Inequality for the hypergeometric distribution.

\begin{lemma}[Hypergeometric Janson Inequality]
  \label{lemma:HJI}
  Suppose that $\{B_i\}_{i \in I}$ is a family of subsets of an $n$-element set $\Omega$, let $m \in \{0, \ldots, n\}$, and let $p = m/n$. Let
  \[
  \mu = \sum_{i \in I} p^{|B_i|} \qquad \text{and} \qquad \Delta = \sum_{i \sim j} p^{|B_i \cup B_j|},
  \]
  where the second sum is over all ordered pairs $(i,j) \in I^2$ such that $i \neq j$ and $B_i \cap B_j \neq \emptyset$. Let $R$ be the uniformly chosen random $m$-subset of $\Omega$ and let $\BB$ denote the event that $B_i \nsubseteq R$ for all $i \in I$. Then for every $q \in [0,1]$,
  \[
  \Pr(\BB) \le 2 \cdot \exp\left(-q\mu + q^2 \Delta / 2\right).
  \]
\end{lemma}

Our main tool in the proof of the $0$-statement will be the following version of the FKG Inequality for the hypergeometric distribution, which gives a lower bound on the probability $\Pr(\BB)$ from the statement of Lemma~\ref{lemma:HJI}. We postpone the easy deductions of Lemmas~\ref{lemma:HJI} and~\ref{lemma:HFKG} from their standard `binomial' versions to Appendix~\ref{sec:omitted-proofs}.

\begin{lemma}[Hypergeometric FKG Inequality]
  \label{lemma:HFKG}
  Suppose that $\{B_i\}_{i \in I}$ is a family of subsets of an $n$-element set $\Omega$. Let $m \in \{0, \ldots, \lfloor n/2 \rfloor\}$, let $R$ be the uniformly chosen random $m$-subset of $\Omega$, and let $\BB$ denote the event that $B_i \nsubseteq R$ for all $i \in I$. Then for every $\eta \in (0,1)$,
  \[
  \Pr(\BB) \ge \, \prod_{i \in I} \left(1 - \left(\frac{(1+\eta)m}{n}\right)^{|B_i|}\right) - \exp\big( - \eta^2m / 4 \big).
  \]
\end{lemma}

Finally, in the proof of Theorem~\ref{thm:main} in the case $m = \exnr - o(n^2)$, we will need the following folklore result from extremal graph theory. As we were unable to find a good reference, we give a proof of this result in Appendix~\ref{sec:omitted-proofs}. Below, $K(n_1, \ldots, n_r)$ denotes the complete $r$-partite graph whose $r$ color classes have sizes $n_1, \ldots, n_r$, respectively.

\begin{lemma}
  \label{lemma:k-Turan}
  For every integer $r \ge 2$ and all integers $n_1, \ldots, n_r$ satisfying $n_1 \le \ldots \le n_r$,
  \[
  \ex\big(K(n_1, \ldots, n_r), K_r\big) = e\big(K(n_1, \ldots, n_r)\big) - n_1n_2.
  \]
\end{lemma}

We remark here that our proof of Lemma~\ref{lemma:k-Turan} shows that the unique extremal graph is obtained by removing all edges joining some two smallest classes.

\subsection{Estimates for binomial coefficients}

\label{sec:estim-binom-coeff}

Throughout the paper, we will often use various estimates for expressions involving binomial coefficients. In this section, we collect some of them for future reference. We start with an easy corollary of Vandermonde's identity.

\begin{lemma}
  \label{lemma:Vandermonde}
  For every $a$, $b$, $c$, and $d$ with $d \le c$,
  \[
  \binom{a}{d} \binom{b}{c-d} \le \binom{a+b}{c}.
  \]
\end{lemma}

Our next lemma estimates the ratio between $\binom{a}{c}$ and $\binom{b}{c}$ for $a, b, c$ satisfying $a > b > c > 0$.

\begin{lemma}
  \label{lemma:acbc}
  If $a > b > c > 0$, then
  \[
  \left(\frac{a}{b}\right)^c \cdot \binom{b}{c} \le \binom{a}{c} \le \left(\frac{a-c}{b-c}\right)^c \cdot \binom{b}{c}.
  \]
\end{lemma}

\subsection{Main tool}

A crucial ingredient in the proof of Theorem~\ref{thm:main} is an estimate of the upper tail of the distribution of the number of edges in a random subhypergraph of a sparse $k$-uniform $k$-partite hypergraph, Lemma~\ref{lemma:d-sets} below. It formalizes the following statement: If some $\HH \subseteq V_1 \times \ldots \times V_k$ contains only a tiny proportion of all the $k$-tuples in $V_1 \times \ldots \times V_k$, then the probability that, for a random choice of $d$-elements sets $W_1 \subseteq V_1, \ldots, W_k \subseteq V_k$, a much larger proportion of $W_1 \times \ldots \times W_k$ falls in $\HH$ decays exponentially in $d$.

\begin{lemma}
  \label{lemma:d-sets}
  For every integer $k$ and all positive $\alpha$ and $\lambda$, there exists a positive $\tau$ such that the following holds. Let $V_1, \ldots, V_k$ be finite sets and let $d$ be an integer satisfying $2 \le d \le \min\{|V_1|, \ldots, |V_k|\}$. Suppose that $\HH \subseteq V_1 \times \ldots \times V_k$ satisfies
  \[
  |\HH| \le \tau \prod_{i = 1}^k |V_i|
  \]
  and that $W_1, \ldots, W_k$ are uniformly chosen random $d$-subsets of $V_1, \ldots, V_k$, respectively. Then,
  \[
  \Pr\left(|\HH \cap (W_1 \times \ldots \times W_k)| > \lambda d^k\right) \le \alpha^d.
  \]
\end{lemma}

We prove Lemma~\ref{lemma:d-sets} in Appendix~\ref{sec:omitted-proofs}. We just remark here that our proof yields that one can take $\tau = (\alpha/2)^{k^2/\lambda} \cdot \lambda^k \cdot d^{-k^3/(d\lambda)}$. (Although the above expression depends on $d$, this dependence is not crucial as $d^{-1/d} \ge e^{-1/e}$ for all $d$. We made the dependence on $d$ above explicit only because $d^{-k^3/(d\lambda)} \ge e^{-1}$ when $d/\log d \ge k^3/\lambda$.)

\section{On $r$-colorable graphs}

\label{sec:r-col-graphs}

Recall that $\Gr$ is the family of all $r$-partite ($r$-colorable) graphs on the vertex set $[n]$ that have exactly $m$ edges and that $\Part$ is the collection of all $r$-colorings of $[n]$ (partitions of $[n]$ into at most $r$ parts). Given a $\Pi \in \Part$, we define $\GP$ to be the family of all $G \in \Gr$ that are properly colored by $\Pi$, that is,
\[
\GP = \{G \in \Gr \colon G \subseteq \Pi\}.
\]
Note that $|\GP| = \binom{e(\Pi)}{m}$. Trivially, we have
\[
\Gr = \bigcup_{\Pi \in \Part} \GP.
\]
We will be particularly interested in \emph{balanced} $r$-colorings, that is, ones where all the color classes have approximately $n/r$ elements. More precisely, given a positive $\gamma$, we let $\Part(\gamma)$ be the family of all partitions of $[n]$ into $r$ parts $V_1, \ldots, V_r$ such that 
\begin{equation}
  \label{eq:Part-gamma}
  \left(\frac{1}{r} - \gamma\right)n \le |V_i| \le \left(\frac{1}{r} + \gamma\right)n \quad \text{for all $i \in [r]$}.
\end{equation}
That is,
\[
\Part(\gamma) = \big\{\{V_1, \ldots, V_r\} \in \Part \colon \text{\eqref{eq:Part-gamma} holds}\big\}.
\]
We can easily neglect colorings that are not balanced in the above sense. The following proposition, originally proved in~\cite{PrSt95}, shows that if $m \gg n$, then almost every graph in $\Gr$ admits only balanced $r$-colorings. The easy proof of Proposition~\ref{prop:unbalanced-graphs} is given in Appendix~\ref{sec:omitted-proofs}.

\begin{prop}
  \label{prop:unbalanced-graphs}
  For every positive $\gamma$, there exists a constant $c$ such that if $m \ge cn$, then
  \[
  \sum_{\Pi \not\in \Part(\gamma)} |\GP| \ll \binom{\exnr}{m} \le |\Gr|.
  \]
\end{prop}

Even though the collections $\GP$ are generally not pairwise disjoint, there is not too much overlap between them. More precisely, if $\Pi$ is not very unbalanced, then, for an overwhelming proportion of all $G \in \GP$, the $r$-coloring $\Pi$ is their unique proper $r$-coloring. The following rigorous version of this statement follows from the (much stronger) results proved in~\cite{PrSt95}.

\begin{thm}
  \label{thm:Gr}
  For every integer $r \ge 2$, every $0 < \gamma \le 1/2r$, and every $m \gg n \log n$,
  \[
  |\Gr| = (1+o(1)) \sum_{\Pi \in \Part(\gamma)} |\GP| = (1+o(1)) \sum_{\Pi \in \Part(\gamma)} \binom{e(\Pi)}{m}.
  \]
\end{thm}

We remark here that Theorem~\ref{thm:Gr} is an easy consequence of Proposition~\ref{prop:unbalanced-graphs}, above, and Proposition~\ref{prop:UP} proved in the next section.

\section{The $0$-statement}

\label{sec:0-statement}

Our aim here is to show that if $(1+\eps)d_r \le m \le (1-\eps)m_r$, then almost all graphs in $\Free$ are not $r$-colorable. As already discussed before, given the (difficult and interesting) result establishing when almost all graphs in $\Gnm$ stop being $r$-colorable~\cite{AcFr99}, we may assume that $m \gg n$. We first give an easy argument that works in the case $n \ll m \ll n^{2-2/(r+2)}$ and then present a more complicated argument that works for all $m$ satisfying $n \log n \ll m \le (1-\eps)m_r$.

\subsection{Counting very sparse $K_{r+1}$-free graphs}

\label{sec:very-sparse}

In this section, generalizing a counting argument of Pr{\"o}mel and Steger~\cite{PrSt96}, we show that if $m$ satisfies $n \ll m \ll n^{2-2/(r+2)}$, then in fact, almost all graphs in $\Free$ have arbitrarily high chromatic number. For our purposes, we only need the statement of Lemma~\ref{lemma:Gnm-count} in the case $k = r$.

\begin{lemma}
  \label{lemma:Gnm-count}
  For every $k \ge 2$, there exist $c > 0$ and $d > 0$ such that
  \[
  |\Free| \gg |\Gk|
  \]
  for all $m$ satisfying $c n \le m \le d n^{2-2/(r+2)}$. 
\end{lemma}
\begin{proof}
  Let $G_{n,m}$ be the uniformly selected random element of $\Gnm$. Clearly,
  \begin{equation}
    \label{eq:Free-Gnm}
    |\Free| = \Pr[\text{$G_{n,m}$ is $K_{r+1}$-free}] \cdot \binom{\binom{n}{2}}{m}.
  \end{equation}
  By Lemma~\ref{lemma:acbc}, we have that for sufficiently large $n$,
  \begin{equation}
    \label{eq:nt-m}
    \binom{\binom{n}{2}}{m} \ge \left(\frac{\binom{n}{2}}{\exnk}\right)^{m} \cdot \binom{\exnk}{m} \ge e^{\frac{m}{k+1}} \cdot \binom{\exnk}{m},
  \end{equation}
  where in the last inequality we used the fact that $\exnk = (1 - \frac{1}{k})\binom{n}{2} + O(n)$. Note that if $m = n^{2-2/(r+2)}$, then
  \[
  n^{r+1} \cdot \left(\frac{m}{n^2}\right)^{\rpt} = m.
  \]
  Since there are fewer than $n^{r+1}$ copies of $K_{r+1}$ in the complete graph on $n$ vertices, the Hypergeometric FKG Inequality (Lemma~\ref{lemma:HFKG}, where we set $\eta = 1/2$) implies that if $m \le dn^{2 - 2/(r+2)}$ for some constant $d$, then
  \begin{equation}
    \label{eq:Gnm-Krp-free}
      \Pr[\text{$G_{n,m}$ is $K_{r+1}$-free}] + \exp\left(-\frac{m}{16}\right) \ge \left(1 - \left(\frac{4m}{n^2}\right)^{\rpt}\right)^{n^{r+1}} \ge \exp\left(-5^{\rpt} d^{\rpt-1}m\right),
  \end{equation}
  provided that $n$ is sufficiently large. Therefore, if $d$ is sufficiently small (i.e., when the right-hand side of~\eqref{eq:Gnm-Krp-free} is larger than $e^{-m/16} + e^{-m/((k+1)(k+2))}$), then by~\eqref{eq:Free-Gnm}, \eqref{eq:nt-m}, and~\eqref{eq:Gnm-Krp-free},
  \[
  |\Free| \ge e^{\frac{m}{k+2}} \cdot \binom{\exnk}{m} \gg k^n \cdot \binom{\exnk}{m} \ge |\Gk|,
  \]
  provided that $m \ge cn$ for a sufficiently large constant $c$.
\end{proof}

\subsection{Counting $K_{r+1}$-free graphs with one monochromatic edge}

\label{sec:one-bad-edge}

In this section, generalizing the approach of Osthus, Pr{\"o}mel, and Taraz~\cite{OsPrTa03}, we count graphs in $\Free$ that are $r$-colorable except for one edge. We show that if $m$ satisfies $n \log n \ll m \le (1-\eps)m_r$, then the number of such graphs is already much larger than $\Gr$. In particular, we shall deduce the following.

\begin{prop}
  \label{prop:one-bad-edge}
  For every $r \ge 2$, every $\eps > 0$, and every $m$ satisfying $n \log n \ll m \le (1-\eps)m_r$,
  \[
  |\Free| \gg |\Gr|.
  \]
\end{prop}

Recall the definitions of $\Part(\gamma)$ and $\GP$ given in Section~\ref{sec:r-col-graphs}. Given a $\Pi \in \Part$, and an edge $e \in \Pic$, we define
\[
\GPe = \{G + e \colon G \in \GmP \}.
\]
Note that $|\GPe| = \binom{e(\Pi)}{m-1}$ and that $\GPe \cap \GPf = \emptyset$ if $e \neq f$. We first show that if $\Pi$ is balanced, then for every edge $e \in \Pic$, the family $\GPe$ contains many $K_{r+1}$-free graphs.

We first set some parameters. Recall that a constant $\eps \in (0,1)$ is given. Let $\gamma$ and $\eta$ be small positive constants such that
\begin{equation}
  \label{eq:eta-gamma-eps}
  \frac{(1+\eta)(1-\eps)}{(1-r\gamma)^2} \le 1 - \frac{\eps}{2} \qquad \text{and} \qquad (1+\gamma r)^{r-1} \le 1 + \frac{\eps}{4}.
\end{equation}
Note also that every $\Pi \in \Part(\gamma)$ satisfies
\begin{equation}
  \label{eq:ePi-lower-0}
  e(\Pi) \ge \binom{r}{2} \cdot \left[\left(\frac{1}{r}-\gamma\right)n\right]^2 = (1-r\gamma)^2\left(1 - \frac{1}{r}\right)\frac{n^2}{2}.
\end{equation}

We are now ready to state and prove our main lemma.

\begin{lemma}
  \label{lemma:GPe-free}
  For all $\Pi \in \Part(\gamma)$ and $m$ with $n \le m \le (1-\eps) m_r$, we have
  \[
  |\GPe \cap \Free| \gg \frac{1}{m} \cdot \binom{e(\Pi)}{m-1}.
  \]
\end{lemma}
\begin{proof}
  Suppose that $\Pi = \{V_1, \ldots, V_r\}$ and that $e$ lies in $V_i$. Let $\KK$ be the collection of (the edge sets of) all copies of $K_{r+1}^-$ induced in $\Pi$ by the two endpoints of the edge $e$ and one vertex from each $V_j$ with $j \neq i$. Let $G_{n,m-1}$ be the random element of $\GmP$ and note that
  \begin{equation}
    \label{eq:GPe-Free}
    |\GPe \cap \Free| = \Pr(\text{$G_{n,m-1} \nsupseteq K$ for every $K \in \KK$}) \cdot \binom{e(\Pi)}{m-1}.
  \end{equation}
  Denote the above probability by $P$. We need to show that $P \gg 1/m$. By the Hypergeometric FKG Inequality (Lemma~\ref{lemma:HFKG}),
  \begin{equation}
    \label{eq:P-FKG}
    P \ge \left(1 - \left(\frac{(1+\eta)(m-1)}{e(\Pi)}\right)^{\rpt-1}\right)^{|\KK|} - \exp\left(-\frac{\eta^2(m-1)}{4}\right).
  \end{equation}
  Observe that by~\eqref{eq:eta-gamma-eps},
  \[
  |\KK| \le \left(\frac{1}{r} + \gamma \right)^{r-1} n^{r-1} = (1 + \gamma r)^{r-1} \cdot \left(\frac{n}{r}\right)^{r-1} \le \left(1 + \frac{\eps}{4}\right) \cdot \left(\frac{n}{r}\right)^{r-1}.
  \]
  Hence, if $n \le m \le n^{2 - 2/(r+2)}$, then $P \ge c$ for some positive constant $c$. Therefore, we may assume that $n^{2-2/(r+2)} \le m \le (1-\eps)m_r$. Recall that $\Pi \in \Part(\gamma)$ and hence by~\eqref{eq:eta-gamma-eps} and \eqref{eq:ePi-lower-0}, recalling the definition of $p_r$ from~\eqref{eq:pr-mr},
  \[
  \frac{(1+\eta)(m-1)}{e(\Pi)} \le \frac{(1+\eta)(1-\eps)m_r}{(1-r\gamma)^2\left(1 - \frac{1}{r}\right)\frac{n^2}{2}} \le \left(1-\frac{\eps}{2}\right)p_r.
  \]
  Therefore, recalling~\eqref{eq:pr-def}, using the fact that $p_r \ll 1$ and $1 - x \ge \exp(-x-x^2)$ if $x \le \frac{1}{2}$, we continue~\eqref{eq:P-FKG} as follows:
  \[
  P + \exp\left(-\frac{\eta^2(m-1)}{4} \right) \ge \exp\left( - \left(1-\frac{\eps}{2}\right) p_r^{\rpt-1} \left(\frac{n}{r}\right)^{r-1} \right) = n^{-\left(1 - \frac{\eps}{2}\right) \cdot \left(2 - \frac{2}{r+2}\right)} \gg \frac{1}{m}.\qedhere
  \]
\end{proof}

Since
\[
\binom{e(\Pi)}{m-1} \ge \frac{m}{e(\Pi)} \cdot \binom{e(\Pi)}{m},
\]
it follows from Lemma~\ref{lemma:GPe-free} that if $n$ is sufficiently large, then
\begin{equation}
  \label{eq:sum-GPe}
  \sum_{\Pi \in \Part(\gamma)} \sum_{e \in \Pic} |\GPe \cap \Free| \gg \sum_{\Pi \in \Part(\gamma)} \frac{e(\Pi^c)}{e(\Pi)} \cdot \binom{e(\Pi)}{m} \ge \frac{1}{2r} \cdot |\Gr|,
\end{equation}
where the last inequality follows from Theorem~\ref{thm:Gr} 
and the fact that $e(\Pi^c) / e(\Pi) \ge 1/(2r-2)$ for every $\Pi \in \Part$, provided that $n \ge 2r$. The left hand side of~\eqref{eq:sum-GPe} counts the pairs $(G, \Pi)$ such that $G \in \GPe \cap \Free$ for some $e \in \Pic$. Therefore, in order to conclude that the number of graphs in $\Free$ with exactly one monochromatic edge is much larger than $|\Gr|$, it is enough to show that for every $\Pi \in \Part(\gamma)$ and every $e \in \Pic$, an overwhelming proportion of all $G \in \GPe$, and hence also an overwhelming proportion of all $G \in \GPe \cap \Free$, do not belong to any other $\GPpf$ with $\Pi' \neq \Pi$ and $f \in \Pipc$.

To this end, given a $\Pi \in \Part$ and $e \in \Pic$, let $\UPe$ be the family of all $G \in \GPe$ for which the pair $(\Pi, e)$ is unique, that is, such that $G \not\in \GPpf$ for any $\Pi' \in \Part$ and $f \in \Pipc$ with $\Pi' \neq \Pi$. Our second key lemma in the proof of the $0$-statement is the following. 

\begin{lemma}
  \label{lemma:GPe-unique}
  For every positive $a$, there exists a constant $c$ such that for all $\Pi \in \Part(\frac{1}{2r})$ and $m \ge c n \log n$ we have
  \[
  |\GPe \setminus \UPe| \le n^{-a} \cdot |\GPe|.
  \]
\end{lemma}

In the proof of Lemma~\ref{lemma:GPe-unique}, we will need an estimate on the number of non-uniquely $r$-colorable graphs.  Given a $\Pi \in \Part$, let $\UP$ be the family of all graphs in $\GP$ for which $\Pi$ is the unique proper $r$-coloring. The following result is implicit in the work of Pr\"omel and Steger~\cite{PrSt95}. For completeness, we give a proof of it in Appendix~\ref{sec:omitted-proofs}.

\begin{prop}
  \label{prop:UP}
  For every positive $a$, there exists a constant $c$ such that for every $\Pi \in \Part(\frac{1}{2r})$, if $m \ge c n \log n$, then
  \[
  |\GP \setminus \UP| \le n^{-a} \cdot |\GP|.
  \]
\end{prop}

\begin{proof}[{Proof of Lemma~\ref{lemma:GPe-unique}}]
  By definition, if $G \in \GPe \setminus \UPe$ then $G \in \GPe \cap \mathcal{G}_m(\Pi',f)$ for some $\Pi \neq \Pi'$ and $f \in \Pipc$. Considering the two cases $e = f$ and $e \neq f$, we infer that either $G - e \in \GmP \setminus \UmP$ or $G -\{e,f\} \in \GmmP \setminus \UmmP$. Let $a' = a+5$ and recall that $|\GPe| = |\GmP| = \binom{e(\Pi)}{m-1}$. By Proposition~\ref{prop:UP}, if $m \ge c n \log n$ for a sufficiently large constant $c$, then
  \[
  |\GPe \setminus \UPe| \le n^{-a'} \cdot \left(|\GmP|+ n^2 \cdot |\GmmP|\right) \le 2n^{-a'+4} \cdot \binom{e(\Pi)}{m-1},
  \]
  where the last inequality holds since $|\GmmP| = \binom{e(\Pi)}{m-2} \le n^2 \cdot \binom{e(\Pi)}{m-1}$.
\end{proof}

\begin{proof}[Proof of Proposition~\ref{prop:one-bad-edge}]
  Finally, if $n \log n \ll m \le (1-\eps)m_r$, then using Lemmas~\ref{lemma:GPe-free} and~\ref{lemma:GPe-unique} we conclude that if $n$ is sufficiently large, then (below $\Pi$ ranges over $\Part(\gamma)$)
  \[
  \begin{split}
    |\Free| & \ge \left| \bigcup_{\Pi} \bigcup_{e \in \Pi^c} \UPe \cap \Free \right| = \sum_{\Pi} \sum_{e \in \Pi^c} \big| \UPe \cap \Free \big| \\
    & \ge \sum_{\Pi} \sum_{e \in \Pi^c} \Big( \big| \GPe \cap \Free \big| - \big| \GPe \setminus \UPe \big| \Big) \\
    & \gg \sum_{\Pi} \sum_{e \in \Pi^c} \left( \frac{1}{m} - \frac{1}{m^2} \right) \cdot \binom{e(\Pi)}{m-1} \ge \frac{1}{3r} \cdot |\Gr|,
  \end{split}
  \]
  where the last inequality follows from Proposition~\ref{prop:unbalanced-graphs}, cf.\ \eqref{eq:sum-GPe}. This completes the proof of the $0$-statement.  
\end{proof}

\section{The $1$-statement}

\label{sec:1-statement}

In the remainder of the paper, we will show that if $m \ge (1+\eps)m_r$ for some positive constant $\eps$, then $|\Frees| \ll |\Free|$, which implies that $|\Free| = (1+o(1)) \cdot |\Gr|$. In this section, we set up some notation and parameters and show that, as already pointed out before, we may restrict our attention to $K_{r+1}$-free graphs that admit a balanced $r$-coloring with few monochromatic edges. This leads to formulating Theorem~\ref{thm:1-statement}, a technical statement formalizing the above claim that graphs which admit a balanced $r$-coloring with few monochromatic edges constitute the vast majority of $\Free$. We then show how Theorem~\ref{thm:1-statement}, together with Theorem~\ref{thm:approx-struct}, implies the $1$-statement of Theorem~\ref{thm:main}. We close the section by defining the split into the sparse and the dense case, which are then handled in Sections~\ref{sec:sparse-case} and~\ref{sec:dense-case}, respectively.

\subsection{Almost $r$-colorability}

\label{sec:almost-r-color}

We begin with a fairly straightforward refinement of Theorem~\ref{thm:approx-struct}, Proposition~\ref{prop:approx-struct} below. Roughly speaking, it says that if $m \gg n^{2-2/(r+1)}$, then not only does almost every $G \in \Free$ admit an $r$-coloring $\Pi$ that makes merely $o(m)$ edges of $G$ monochromatic, but moreover, for almost every such $G$, every $\Pi$ with this property has $r$ parts of size $n/r + o(n)$.

\begin{prop}
  \label{prop:approx-struct}
  For every positive integer $r$ and all positive $\gamma$ and $\delta$, there exists a $C$ such that if $m \ge Cn^{2-\frac{2}{r+2}}$, then almost every $G \in \Free$ admits a partition $\Pi = \{V_1, \ldots, V_r\}$ of $[n]$ such that
  \begin{equation}
    \label{eq:approx-struct-edges}
    e(G \setminus \Pi) = \sum_{i=1}^r e_G(V_i) \le \delta m.
  \end{equation}
  Moreover, if $\delta$ is sufficiently small as a function of $\gamma$, then for almost all $G \in \Free$, every $\Pi$ satisfying~\eqref{eq:approx-struct-edges} also satisfies~\eqref{eq:Part-gamma}, that is, belongs to $\Part(\gamma)$.
\end{prop}
\begin{proof}
  Without loss of generality, we may assume that $\delta$ is sufficiently small as a function of $\gamma$. In particular, we may assume that it satisfies
  \begin{equation}
    \label{eq:gamma-delta}
    \delta \cdot \log \frac{4}{\gamma^2} - (1-\delta) \frac{\gamma^2}{2} < - \frac{\gamma^2}{3}.
  \end{equation}
  The existence of a partition $\Pi \in \Part$ satisfying~\eqref{eq:approx-struct-edges} for almost all $G \in \Free$ follows directly from Theorem~\ref{thm:approx-struct}. Therefore, it suffices to count graphs $G \in \Free$ that admit a partition that satisfies~\eqref{eq:approx-struct-edges} but not~\eqref{eq:Part-gamma}. To this end, fix an arbitrary partition $\Pi \in \Part$ that does not satisfy~\eqref{eq:Part-gamma} and observe that $e(\Pi)$ is maximized when one of the parts has size $\lfloor(\frac{1}{r} + \gamma)n+1\rfloor$ or $\lceil(\frac{1}{r}-\gamma)n-1\rceil$ and the sizes of the remaining parts are as equal as possible. Therefore,
  \[
  \binom{n}{2} - \min\left\{ \binom{\left(\frac{1}{r} - \gamma\right)n}{2} + (r-1)\binom{\left(\frac{1}{r} + \frac{\gamma}{r-1}\right)n}{2}, \binom{\left(\frac{1}{r} + \gamma\right)n}{2} + (r-1)\binom{\left(\frac{1}{r} - \frac{\gamma}{r-1}\right)n}{2} \right\}
  \]
  is an upper bound on $e(\Pi)$. It follows that
  \begin{equation}
    \label{eq:ePi-gamma-upper}
    e(\Pi) \le \binom{n}{2} - r \binom{\frac{n}{r}}{2} - \frac{\gamma^2 r}{2(r-1)}n^2 \le \exnr - \frac{\gamma^2 n^2}{2}.
  \end{equation}
  Note that the number $N_\Pi$ of graphs $G \in \Free$ for which~\eqref{eq:approx-struct-edges} holds for our fixed partition~$\Pi$ satisfies
  \begin{equation}
    \label{eq:NPi-one}
    N_\Pi \le \sum_{t = 0}^{\delta m} \binom{\binom{n}{2}}{t} \binom{e(\Pi)}{m-t} \le m \cdot \binom{\binom{n}{2}}{\delta m} \binom{e(\Pi)}{m - \delta m},
  \end{equation}
  where the second inequality holds because the summand in~\eqref{eq:NPi-one} is an increasing function of $t$ on the interval $[0,m/2]$. Now, using our bound on $e(\Pi)$ for $\Pi$ that do not satisfy~\eqref{eq:Part-gamma}, we have
  \begin{equation}
    \label{eq:NPi-two}
    \begin{split}
      N_\Pi & \le m \cdot \binom{\binom{n}{2}}{\delta m} \binom{\exnr - \frac{\gamma^2 n^2}{2}}{m - \delta m} \\
      & \le m \cdot \left(\frac{\binom{n}{2}-\delta m}{\frac{\gamma^2n^2}{4}-\delta m}\right)^{\delta m} \binom{\frac{\gamma^2n^2}{4}}{\delta m} \cdot \left(\frac{\exnr - \frac{\gamma^2n^2}{2}}{\exnr-\frac{\gamma^2n^2}{4}}\right)^{m-\delta m} \binom{\exnr - \frac{\gamma^2n^2}{4}}{m-\delta m} \\
      & \le m \cdot \left(\frac{4}{\gamma^2}\right)^{\delta m} \cdot \left(1 - \frac{\gamma^2}{2}\right)^{(1-\delta) m} \cdot \binom{\exnr}{m} \le  \exp\left(-\frac{\gamma^2m}{4}\right) \cdot \binom{\exnr}{m},
    \end{split}
  \end{equation}
  where the second inequality follows from Lemma~\ref{lemma:acbc} (applied twice), the third inequality follows from Lemma~\ref{lemma:Vandermonde}, and the last inequality follows from~\eqref{eq:gamma-delta}, provided that $n$ is sufficiently large (and, consequently, $m$ is sufficiently large). Finally, the result follows from~\eqref{eq:NPi-two} since there are at most $r^n$ partitions $\Pi \in \Part$ and at least $\binom{\exnr}{m}$ graphs in $\Free$.
\end{proof}

In view of Proposition~\ref{prop:approx-struct}, for positive $\gamma$ and $\delta$, let $\Freedg$ be the collection of graphs $G \in \Free$ that satisfy~\eqref{eq:approx-struct-edges} for some $\Pi \in \Part(\gamma)$ and no $\Pi \not\in \Part(\gamma)$. In other words, $\Freedg$ is the collection of graphs $G \in \Free$ that are almost $r$-colorable (i.e., admit an $r$-coloring which makes only at most $\delta m$ edges of $G$ monochromatic) and such that every $r$-coloring $\Pi$ that makes only at most $\delta m$ edges of $G$ monochromatic has color classes of sizes only between $(1/r-\gamma)n$ and $(1/r+\gamma)n$. In this notation, Proposition~\ref{prop:approx-struct} says that almost all graphs in $\Free$ belong to $\Freedg$ provided that $\delta$ is sufficiently small as a function of $\gamma$ and $m \ge C_{\ref{prop:approx-struct}}(\delta, \gamma) \cdot n^{-\frac{2}{r+2}}$.

\begin{claim}
  \label{claim:unfriendly}
  For every integer $r$ and all positive $\gamma$ and $\delta$, every $G \in \Freedg$ admits a $\Pi \in \Part(\gamma)$ that satisfies~\eqref{eq:approx-struct-edges} and such that each vertex of $G$ has at least as many neighbors in each color class of $\Pi$ as in its own class, that is, letting $\Pi = \{V_1, \ldots, V_r\}$,
\begin{equation}
  \label{eq:Pi-unfriendly}
  \deg_G(v,V_i) \le \min_{j \neq i} \deg_G(v, V_j) \quad \text{for all $i \in [r]$ and $v \in V_i$}.
\end{equation}
\end{claim}
\begin{proof}
  To see this, given such a $G$, let $\Pi \in \Part$ be a partition that minimizes $e(G \setminus \Pi)$ and suppose that $\Pi = \{V_1, \ldots, V_r\}$. The minimality of $\Pi$ immediately implies~\eqref{eq:Pi-unfriendly}. Indeed, if there were $i, j \in [r]$ and $v \in V_i$ such that $\deg_G(v, V_i) > \deg_G(v, V_j)$, then the partition $\Pi'$ obtained from $\Pi$ by moving the vertex $v$ from $V_i$ to $V_j$ would satisfy $e(G \setminus \Pi') < e(G \setminus \Pi)$, contradicting the minimality of $\Pi$. Moreover, since $e(G \setminus \Pi) \le \delta m$ by the definition of $\Freedg$ and the minimality of $\Pi$, then $\Pi \in \Part(\gamma)$, again by the definition of $\Freedg$.
\end{proof}

In view of the above, given positive constants $\gamma$ and $\delta$ and a balanced $r$-coloring $\Pi \in \Part(\gamma)$, let
\[
\Freedgp = \big\{G \in \Freedg \colon \text{$(G, \Pi)$ satisfy~\eqref{eq:approx-struct-edges} and~\eqref{eq:Pi-unfriendly}} \big\}.
\]
By Claim~\ref{claim:unfriendly}, we have
\begin{equation}
  \label{eq:Freedg-partition}
  \Freedg = \bigcup_{\Pi \in \Part(\gamma)} \Freedgp.
\end{equation}

Next, let us break down the family $\Frees$ of non-$r$-colorable $K_{r+1}$-free graphs with respect to the above partition of $\Free$. First, let
\[
\Freesdg = \Freedg \setminus \Gr,
\]
then let
\[
\Freesdgp = \big\{G \in \Freedgp \colon e(G \setminus \Pi) > 0\big\},
\]
and note that, by~\eqref{eq:Freedg-partition},
\begin{equation}
  \label{eq:Freesdg-partition}
  \Freesdg \subseteq \bigcup_{\Pi \in \Part(\gamma)} \Freesdgp.
\end{equation}
(Note that we cannot write an equality in~\eqref{eq:Freesdg-partition} since the fact that $G \in \Freesdgp$ for some $\Pi$ does not mean that $G$ is not $r$-colorable).

Finally, since under the assumption that $m \ge (1+\eps)m_r \gg n^{2 - \frac{2}{r+2}}$, Proposition~\ref{prop:approx-struct} applies with arbitrarily small $\gamma$ and $\delta$, it is enough to prove the following theorem, which is the essence of the $1$-statement of Theorem~\ref{thm:main}.

\begin{thm}
  \label{thm:1-statement}
  For every integer $r$ and every positive $\eps$, there exist a positive constant $\gamma$ and a function $\omega$ satisfying $\omega(n) \to \infty$ as $n \to \infty$ such that the following holds for all sufficiently small positive $\delta$. For every $n$, if $m \ge (1+\eps)m_r$, then
  \begin{equation}
    \label{eq:Freesdgp}
    |\Freesdgp| \le \frac{1}{\omega(n)} \cdot \binom{e(\Pi)}{m} \quad \text{for every $\Pi \in \Part(\gamma)$}.
  \end{equation}
\end{thm}

Indeed, Theorem~\ref{thm:Gr} and Proposition~\ref{prop:approx-struct} together with \eqref{eq:Freesdg-partition} and \eqref{eq:Freesdgp} imply that
\begin{align*}
  |\Frees| & \le |\Free \setminus \Freedg| + \sum_{\Pi \in \Part(\gamma)} |\Freesdgp| \\
  & = o\big(|\Free|\big) + \sum_{\Pi \in \Part(\gamma)} \frac{1}{\omega(n)} \cdot \binom{e(\Pi)}{m} \\
  & \le o\big(|\Free|\big) + \frac{1}{\omega(n)} \cdot (1+o(1)) \cdot |\Gr| \ll |\Free|.
\end{align*}
In the remainder of the paper, we will prove Theorem~\ref{thm:1-statement}.

\subsection{Parameters}

\label{sec:parameters-main}

We now choose some parameters. Recall that an integer $r$ and a positive constant $\eps$ are given. We may clearly assume that $\eps \le 1$. We first define the split between the sparse and the dense cases, see Section~\ref{sec:sketch-proof-1}. To this end, we let
\begin{equation}
  \label{eq:xi}
  \xi = \left( \frac{2^9 e}{3} \right)^{-30r^3}.
\end{equation}
Next, we choose small positive $\rho$ and $\gamma$ with $\gamma < \frac{1}{9r}$ so that, letting
\begin{equation}
  \label{eq:crgammap}
  \zeta = \frac{1+\eps}{1-\rho} \cdot \left(\frac{1}{1-r\gamma}\right)^{r-1},
\end{equation}
we have
\begin{equation}
  \label{eq:zeta}
  \zeta \le 2^r \quad \text{and} \quad \left(1+\frac{3\eps}{4}\right)^{\rpt - 1} \ge \left(1+\frac{\eps}{3}\right) \cdot \zeta.
\end{equation}
For example, we may take $\rho = \frac{\eps}{20}$ and $\gamma = \frac{\eps}{9r+9}$. Our assumption that $\gamma < \frac{1}{9r}$ guarantees that for every $\Pi \in \Part(\gamma)$,
\begin{equation}
  \label{eq:ePi-lower}
  e(\Pi) \ge \binom{r}{2} \cdot \left[\left(\frac{1}{r} - \gamma\right)n\right]^2 \ge \binom{r}{2} \cdot \left(\frac{8n}{9r}\right)^2 \ge \frac{3n^2}{16}.
\end{equation}
Finally, let us also assume that we have chosen a small positive constant $\delta$. Since this constant will have to satisfy a series of inequalities that use parameters that have not yet been introduced (but all depend only on the quantities defined so far), we will make this choice more specific somewhat later in the proof. In particular, we assume that $\delta \le \delta_{\ref{prop:approx-struct}}(\gamma)$.

\subsection{Setup}

\label{sec:setup}

For the remainder of the proof, let us fix some $\Pi = \{V_1, \ldots, V_r\} \in \Part(\gamma)$, assume that $m \ge (1+\eps)m_r$, and let
\[
\Fs = \Freesdgp.
\]
Recall that our goal is to prove Theorem~\ref{thm:1-statement}, i.e., that
\[
|\Fs| \le \frac{1}{\omega(n)} \cdot \binom{e(\Pi)}{m}
\]
for some function $\omega$ satisfying $\omega(n) \to \infty$ that does not depend on $\Pi$. We need a few more pieces of notation. Given a $G \in \Fs$, for each $i \in [r]$, we let $T_i(G)$ be the subgraph of $G$ induced by $V_i$, the $i$th color class of $\Pi$. Moreover, we let $T(G)$ be the subgraph of all monochromatic edges of $G$ (in the coloring $\Pi$), that is, $T(G) = T_1(G) \cup \ldots \cup T_r(G)$. As we pointed out in Section~\ref{sec:sketch-proof-1}, our general strategy will be to partition the family $\Fs$ into several classes according to the distribution of edges in the graphs $T(G)$, show that each of these classes is small, and then deduce that $|\Fs|$ is small.

Recall that $\Pi \in \Part(\gamma)$ is fixed. Let $\TT$ denote the collection of all graphs consisting of at most $\delta m$ monochromatic (in the coloring $\Pi$) edges. That is, let $\TT$ be the set of all graphs $T \subseteq \Pic$ with at most $\delta m$ edges. Given a $T \in \TT$, let
\[
\Fst = \{G \in \Fs \colon T(G) = T\}.
\]

As pointed out in Section~\ref{sec:sketch-proof-1}, we will use completely different arguments to handle the cases $m \le e(\Pi) - \xi n^2$ (the sparse case) and $m > e(\Pi) - \xi n^2$ (the dense case). We begin with the main, much harder, case $m \le e(\Pi) - \xi n^2$. The other case is addressed in Section~\ref{sec:dense-case}.

\section{The sparse case ($m \le e(\Pi) - \xi n^2$)}

\label{sec:sparse-case}

\subsection{More parameters}

First, we need to define three more parameters that will play central roles in our proof. First, let
\begin{equation}
  \label{eq:nu}
  \nu = \frac{\rho}{(2r)^{2r+1}}
\end{equation}
and
\begin{equation}
  \label{eq:D}
  D = \nu \cdot \frac{m}{n \log n}.
\end{equation}
For the sake of clarity of presentation, we will assume that $D$ is an integer. Next, let $\beta$ be a small positive constant satisfying
\begin{equation}
  \label{eq:beta}
  \left(\frac{2e}{\xi \beta}\right)^{\beta m/n} \le m^{D/2}.
\end{equation}
Note that choosing such $\beta$ is possible, since $m^D = \exp(\Omega(m/n))$ and $(\frac{2e}{\xi\beta})^\beta \to 1$ as $\beta \to 0$. Also, observe that $D \ll \beta m / n$.

\subsection{Setup}

\label{sec:setup-sparse}

Recall the definition of $\TT$ from Section~\ref{sec:setup}. Let us fix a $T \in \TT$. Let $U(T)$ be some (canonically chosen) edge-maximal subgraph of $T$ with maximum degree at most $D$. Let $X(T)$ be the set of vertices that have the maximum allowed degree in $U(T)$, that is, the set of all $v$ whose degree in $U(T)$ is $D$. Observe that
\begin{equation}
  \label{eq:UT-lower}
  e(U(T)) \ge e(T - X(T)) + |X(T)| \cdot D/2,
\end{equation}
since, by the maximality of $U(T)$, every edge of $T \setminus U(T)$ has at least one endpoint in $X(T)$. 

For every $i \in [r]$, let $U_i(T)$ be the subgraph of $U(T)$ induced by the set $V_i$, the $i$th color class of $\Pi$, and let $X_i(T) = X(T) \cap V_i$. Finally, let $H(T) \subseteq X(T)$ denote the set of vertices $v$ in $X(T)$ whose degree in $T$ is at least $\beta m / n$. We will refer to vertices in $H(T)$ as the vertices with high degree in~$T$. We split the family $\TT$ according to whether the inequality
\begin{equation}
  \label{eq:low-high-sparse}
  |H(T)| \le \frac{\eps \xi}{6} \cdot \frac{n \log m}{m} \cdot e(U(T))
\end{equation}
does or does not hold. More precisely, we let $\TTL$ be the collection of all $T \in \TT$ for which~\eqref{eq:low-high-sparse} holds and let $\TTH = \TT \setminus \TTL$. We will separately count the graphs in $\Fst$ with $T \in \TTL$ (we will refer to it as the \emph{low degree case}) and $T \in \TTH$ (this will be referred to as the \emph{high degree case}).

\subsection{Recap of the proof outline}

There will be four main ingredients in our proof. First, in Section~\ref{sec:counting-graphs}, in Lemmas~\ref{lemma:T-count-basic} and \ref{lemma:T-count}, we will count the graphs $T \in \TT$ with particular values of $e(T)$, $e(T - X(T))$, $|X(T)|$, and $|H(T)|$; this is relatively straightforward. Second, in Section~\ref{sec:bounding-Fst-UT}, in Lemma~\ref{lemma:Janson}, using the Hypergeometric Janson Inequality (Lemma~\ref{lemma:HJI}), we will give an upper bound on the size of $\Fst$ as a function of $e(U(T))$. These three lemmas will already be enough to prove that the number of graphs $G \in \Fs$ that fall into the low degree case (i.e., $T(G) \in \TTL$) is at most $m^{-\eps/4} \binom{e(\Pi)}{m}$, see Lemma~\ref{lemma:low-degree} in Section~\ref{sec:sparse-low-degree-case}. In order to count the graphs $G \in \Fs$ that fall into the high degree case (i.e., $T(G) \in \TTH$), we will have to further split them into two classes, which we term the \emph{regular} and \emph{irregular cases}, depending on the distribution of the neighborhoods of the vertices in $H(T(G))$. We will make this division precise in Section~\ref{sec:sparse-high-degree-case}. The third ingredient in our proof, Lemmas~\ref{lemma:good-H-basic} and \ref{lemma:good-H}, together with Lemmas~\ref{lemma:T-count-basic} and \ref{lemma:T-count}, provides an upper bound on the number of graphs in $\bigcup_T \Fst$, where the union is taken over all $T$ that fall into the regular case, see Section~\ref{sec:regular-case}. Finally, in Section~\ref{sec:irregular-case}, we will use Lemmas~\ref{lemma:d-sets} and \ref{lemma:T-count-basic} to bound the number of graphs that fall into the irregular case. Counting the graphs that fall into the irregular case with the use of Lemma~\ref{lemma:d-sets} is the main technical novelty of this paper.

\subsection{Counting the graphs in $\TT$}

\label{sec:counting-graphs}

For an integer $t$ with $1 \le t \le \delta m$, let $\TT_t$ be the subfamily of $\TT$ consisting of graphs with exactly $t$ edges. Since we are going to treat differently graphs $T \in \TT$ with different values of $e(T)$, $e(U(T))$, $|X(T)|$, and $|H(T)|$, let us further subdivide the families $\TT_t$. Even though the forthcoming definitions might seem somewhat odd at first, they will be very convenient to work with later in the proof. For integers $t^*$, $x$, and $h$, we let $\TT_t(t^*,x,h)$ be the subfamily of $\TT_t$ consisting of all graphs $T$ for which there exist sets $H, X \subseteq [n]$ with $|H| = h$, $|X| = x$, and $H \subseteq X$ such that:
\begin{enumerate}[(i)]
\item
  \label{item:TT-1}
  $e(T - X) = t^*$, that is, $T$ has exactly $t^*$ edges which have no endpoint in $X$,
\item
  $\deg_T(v) < \beta m/n$ for every $v \not\in H$.
\end{enumerate}
Moreover, let $\TT_t'(t^*,x,h)$ be the subfamily of $\TT_t(t^*,x,h)$ consisting of graphs that additionally satisfy
\begin{enumerate}[(i)]
\setcounter{enumi}{2}
\item
  \label{item:TT-3}
  $\deg_T(v) \ge \beta m / n$ for every $v \in H$.
\end{enumerate}
Since every $T \in \TT$ satisfies (\ref{item:TT-1})--(\ref{item:TT-3}) above with $t^* = E(T-X(T))$, $X = X(T)$, and $H = H(T)$, it follows that
\[
T \in \TT_{e(T)}\big(e(T - X(T)), |X(T)|, |H(T)|\big) \subseteq \TT_{e(T)}.
\]
We shall now prove upper bounds on the sizes of the families $\TT_t$ and $\TT_t(t^*,x,h)$. We remark that the somewhat strange-looking form of these bounds will be very convenient for their later applications.

\begin{lemma}
  \label{lemma:T-count-basic}
  If $t \le \delta m$, then
  \[
  |\TT_t| \cdot \binom{e(\Pi)}{m - t} \le \left(\frac{e}{\xi \delta}\right)^{\delta m} \cdot \binom{e(\Pi)}{m}.
  \]
\end{lemma}

\begin{lemma}
  \label{lemma:T-count}
  For all integers $m'$, $t$, $t^*$, $x$, and $h$ with $m' \le m$,
  \[
  |\TT_t(t^*,x,h)| \cdot \binom{e(\Pi)}{m'-t} \le e^{1/\xi} \cdot m^{t^* + xD/2} \cdot \exp\left(\frac{2mh}{\xi n}\right) \cdot \binom{e(\Pi)}{m'}.
  \]
\end{lemma}

\begin{proof}[{Proof of Lemma~\ref{lemma:T-count-basic}}]
  We use the trivial bound
  \begin{equation}
    \label{eq:TTt}
    |\TT_t| \le \binom{e(\Pic)}{t}.
  \end{equation}
  We then use the identity
  \[
  \frac{\binom{e(\Pic)}{t} \binom{e(\Pi)}{m-t}}{\binom{e(\Pi)}{m}} = \prod_{s=0}^{t-1} \frac{\binom{e(\Pic)}{s+1} \binom{e(\Pi)}{m-s-1}}{\binom{e(\Pic)}{s} \binom{e(\Pi)}{m-s}} = \prod_{s=0}^{t-1} \left( \frac{e(\Pic)-s}{s+1} \cdot \frac{m-s}{e(\Pi)-m+s+1} \right)
  \]
  to deduce that, since $m \le e(\Pi) - \xi n^2$,
  \begin{equation}
    \label{eq:TTt-ratio}
    \frac{\binom{e(\Pic)}{t} \binom{e(\Pi)}{m-t}}{\binom{e(\Pi)}{m}} \le \prod_{s=0}^{t-1} \left( \frac{n^2}{s+1} \cdot \frac{m}{\xi n^2} \right) = \frac{1}{t!} \cdot \left(\frac{m}{\xi}\right)^t \le \left(\frac{em}{\xi t}\right)^t \le \left(\frac{e}{\xi \delta}\right)^{\delta m},
  \end{equation}
  where we used the fact that $t! \ge (t/e)^t$ and that the function $t \mapsto (\frac{em}{\xi t})^t$ is increasing on the interval $(0,\delta m]$, as $\delta, \xi \le 1$.
\end{proof}

\begin{proof}[{Proof of Lemma~\ref{lemma:T-count}}]
  We prove the lemma by induction on $x$. For the induction base, the case $x = 0$, note that if $x = 0$, then (in order for the family $\TT_t(t^*,x,h)$ to be non-empty) we must have $h = 0$ and $t^* = t$. Since $\TT_t(t^*,x,h) \subseteq \TT_t$, it now follows from \eqref{eq:TTt} and \eqref{eq:TTt-ratio}, with $m$ replaced by $m'$, that 
 \[
 |\TT_t(t^*,x,h)| \cdot \binom{e(\Pi)}{m'-t} \cdot \binom{e(\Pi)}{m'}^{-1} \le \left(\frac{em'}{\xi t}\right)^t \le \left(\frac{em}{\xi t}\right)^t = \left(\frac{e}{\xi t^*}\right)^{t^*} m^{t^*} \le e^{1/\xi} \cdot m^{t^*},
 \]
 where the last inequality holds because the function $t^* \mapsto (\frac{e}{\xi t^*})^{t^*}$ is maximized when $t^* = 1/\xi$.

 \smallskip
 Assume now that $x \ge 1$. Given a $T \in \TT_t(t^*,x,h)$, we fix some $X$ and $H$ from the definition of $\TT_t(t^*,x,h)$, pick an arbitrary vertex $v \in X$. Next, let $d = \deg_T(v)$ and obtain a subgraph $T' \subseteq T$ by removing all $d$ edges incident to $v$. Clearly, $T'$ lies in $\TT_{t-d}(t^*,x-1,h) \cup \TT_{t-d}(t^*,x-1,h-1)$. Moreover, if $d > \beta m / n$, then necessarily $v \in H$ (but not vice versa!) and consequently $T' \in \TT_{t-d}(t^*,x-1,h-1)$. It follows that
  \begin{equation}
    \label{eq:T-count-step}
    |\TT_t(t^*,x,h)| \le \sum_{d=0}^{\beta m/n} n\binom{n}{d} |\TT_{t-d}(t^*,x-1,h)| + \sum_{d=0}^n n \binom{n}{d} |\TT_{t-d}(t^*,x-1,h-1)|.
  \end{equation}
  Since $t \le m' \le m \le e(\Pi) - \xi n^2$, then
  \begin{equation}
    \label{eq:T-count-step-ratio}
    \begin{split}
      \frac{\binom{n}{d} \binom{e(\Pi)}{m'-t}}{\binom{e(\Pi)}{m'-t+d}} & = \prod_{s=0}^{d-1} \frac{\binom{n}{s+1} \binom{e(\Pi)}{m'-t+d-s-1}}{\binom{n}{s} \binom{e(\Pi)}{m'-t+d-s}} = \prod_{s=0}^{d-1} \left(\frac{n-s}{s+1} \cdot \frac{m'-t+d-s}{e(\Pi)-m'+t-d+s+1}\right) \\
      & \le \prod_{s=0}^{d-1} \left(\frac{n}{s+1} \cdot \frac{m'}{\xi n^2}\right) = \frac{1}{d!} \cdot \left(\frac{m'}{\xi n}\right)^d \le \left(\frac{em}{\xi nd}\right)^d,
    \end{split}
  \end{equation}
  where we again used the fact that $d! \ge (d/e)^d$. Recall that for every positive $a$, the function $x \mapsto (a/x)^x$ is increasing on the interval $(0,a/e]$ and decreasing on the interval $[a/e, \infty)$. Hence, by~\eqref{eq:T-count-step-ratio},
  \begin{equation}
    \label{eq:T-count-step-low-deg}
    \sum_{d=0}^{\beta m/n} n\binom{n}{d} \frac{\binom{e(\Pi)}{m'-t}}{\binom{e(\Pi)}{m'-t+d}} \le \sum_{d=0}^{\beta m /n} n \left(\frac{em}{\xi nd}\right)^d \le n^2 \left(\frac{e}{\xi \beta}\right)^{\beta m/n} \le \frac{1}{2} \left(\frac{2e}{\xi\beta}\right)^{\beta m/n} \le \frac{1}{2} m^{D/2},
  \end{equation}
  where the last inequality follows from~\eqref{eq:beta}, and
  \begin{equation}
    \label{eq:T-count-step-high-deg}
    \sum_{d=0}^n n\binom{n}{d} \frac{\binom{e(\Pi)}{m'-t}}{\binom{e(\Pi)}{m'-t+d}} \le \sum_{d=0}^{n} n \left(\frac{em}{\xi nd}\right)^d \le n^2 \cdot \exp\left(\frac{m}{\xi n}\right)  \le \frac{1}{2} \exp\left(\frac{2m}{\xi n}\right).
  \end{equation}
  The claimed bound follows easily from the inductive assumption, \eqref{eq:T-count-step}, \eqref{eq:T-count-step-low-deg}, and \eqref{eq:T-count-step-high-deg}.
\end{proof}

\subsection{Bounding $|\Fst|$ in terms of $e(U(T))$}

\label{sec:bounding-Fst-UT}

We shall now state and prove our main lemma for the low degree case. It provides an upper bound on the size of $\Fst$ in terms of the number of edges in the graph $U(T)$. The lemma follows the natural and simple idea described in Section~\ref{sec:mr-threshold}, which was already exploited in~\cite{OsPrTa03} in the case $r = 2$. If $m \ge (1+\eps)m_r$, then, at least under all the simplifying assumptions made in Section~\ref{sec:mr-threshold}, the proportion $P$ of graphs in $\Fst$ with exactly one monochromatic edge is asymptotically smaller than $\frac{1}{m}$. Unfortunately, the calculation that we used to estimate $P$ is merely some intuition to keep in mind, as in reality things are considerably more complicated. Whereas the intuition that avoiding different copies of $K_{r+1}^-$ in $\Pi$, whose missing edges belong to $T$, can be treated as independent events is valid and can be made rigorous when the graph $T$ is small, it is no longer right when $T$ becomes large. In fact, it turns out that considering more copies of $K_{r+1}^-$ when we apply the Hypergeometric Janson Inequality (Lemma~\ref{lemma:HJI}) does not necessarily improve the bound, but can actually worsen it. This is why we work with the subgraph $U(T)$ of $T$ with bounded maximum degree. Still, our biggest problem is that the best bound for $|\Fst|$ that we can obtain using the Hypergeometric Janson Inequality is not sufficiently strong to compensate for having to sum it over all $T \in \TT$. This is why we split into the low degree and the high degree cases and are forced to use different methods to handle the high degree case.

Our main lemma in the low degree case is the following.

\begin{lemma}
  \label{lemma:Janson}
  For every $T \in \TT$,
  \[
  |\Fst| \le 2 m^{-(1+\eps) \cdot e(U(T))} \cdot \binom{e(\Pi)}{m - e(T)}.
  \]
\end{lemma}

In the next section, we show that Lemma~\ref{lemma:Janson}, together with Lemmas~\ref{lemma:T-count-basic} and~\ref{lemma:T-count}, resolves the low degree case, that is, that it implies that
\[
|\{G \in \Fs \colon T(G) \in \TTL\}| \le m^{-\eps/4} \cdot \binom{e(\Pi)}{m},
\]
cf.\ Theorem~\ref{thm:1-statement}. In the remainder of this section, we prove the lemma.

\begin{proof}[Proof of Lemma~\ref{lemma:Janson}]
  In order to prove the lemma, given a $T \in \TT$, we will count the number of graphs $G' \subseteq \Pi$ with $m - e(T)$ edges such that $G = G' \cup T$ is $K_{r+1}$-free. The crucial observation is that for every such $G'$ and every edge $\{v, w\} \in U_i(T)$, none of the $\prod_{j \neq i} |V_j|$ copies of $K_{r+1}^-$ in $\Pi$ induced by $v$, $w$, and one vertex in each $V_j$ with $j \neq i$ can be fully contained in $G'$. In the remainder of the proof, we will use the Hypergeometric Janson Inequality to count graphs $G'$ satisfying this constraint. Note that
  \begin{equation}
    \label{eq:m_0}
    e(U(T)) \le Dn = \nu \cdot \frac{m}{\log n} = \frac{\rho}{(2r)^{2r+1}} \cdot \frac{m}{\log n},
  \end{equation}
since $U(T)$ has maximum degree at most $D$. 
  
  Let $\KK$ be the collection of (the edge sets of) all copies of $K_{r+1}^-$ induced in $\Pi$ by the two endpoints of some edge in $U_i(T)$ and one vertex in each $V_j$ with $j \neq i$. Given $(K_1, K_2) \in \KK^2$, we write $K_1 \sim K_2$ to denote the fact that $K_1$ and $K_2$ share at least one edge but $K_1 \neq K_2$. Let $p = \frac{m - e(T)}{e(\Pi)}$ and let
  \[
  \mu = \sum_{K \in \KK} p^{e(K)} \qquad \text{and} \qquad \Delta = \sum_{K_1 \sim K_2} p^{e(K_1 \cup K_2)},
  \]
  where the second sum above is over all ordered pairs $(K_1, K_2) \in \KK^2$ such that $K_1 \sim K_2$. By the Hypergeometric Janson Inequality, Lemma~\ref{lemma:HJI}, for every $q \in [0, 1]$,
  \[
  |\Fst| \le 2 \cdot \exp\left(-q\mu + q^2 \Delta /2\right) \cdot \binom{e(\Pi)}{m - e(T)}.
  \]
  Therefore, it suffices to show that for some $q \in [0,1]$, we have
  \begin{equation}
    \label{eq:q-goal}
    q\mu - q^2 \Delta/2 \ge (1+\eps) \log m \cdot e(U(T)),
  \end{equation}
  which we will do in the remainder of the proof of the lemma.

  \medskip
  \noindent
  \textbf{Estimating $\mu$ and $\Delta$.}
  Recall that our fixed partition $\Pi$ lies in $\Part(\gamma)$ and hence $|V_i| = (1/r \pm \gamma)n$ for each $i \in [r]$. It follows that
  \begin{equation}
    \label{eq:mu}
    \mu = |\KK| \cdot p^{\rpt-1} \ge e(U(T)) \cdot \left(\frac{1}{r} - \gamma\right)^{r-1} n^{r-1} \cdot p^{\rpt-1}.
  \end{equation}
  With the aim of estimating $\Delta$, for every $s \in [r-2]$, let
  \[
  N_s = \max \left\{ \prod_{i \in I} |V_i| \colon I \subseteq [r] \text{ with } |I| = s \right\} \le \left(\frac{1}{r} + \gamma\right)^s n^s.
  \]
  Let us now fix two edges $v_1w_1$ and $v_2w_2$ of $U(T)$ and compute the contribution to $\Delta$ of all ordered pairs $(K_1, K_2) \in \KK$ such that $K_1 \sim K_2$ and $v_1w_1$ and $v_2w_2$ are the missing edges in $K_1$ and $K_2$, respectively. We denote these contributions by:
  \begin{itemize}
  \item
    $\Delta_1$ when $v_1w_1$ and $v_2w_2$ lie in the same color class and are disjoint,
  \item
    $\Delta_2$ when $v_1w_1$ and $v_2w_2$ lie in the same color class and share exactly one endpoint,
  \item
    $\Delta_3$ when $v_1w_1 = v_2w_2$, and
  \item
    $\Delta_4$ when $v_1w_1$ and $v_2w_2$ lie in different color classes.
  \end{itemize}
  A moment's thought reveals that
  \begin{align}
    \label{eq:Delta1}
    \Delta_1 & \le \sum_{s=2}^{r-1} \rms N_s N_{r-s-1}^2 p^{2\rpt - \st - 2}, \\
    \label{eq:Delta2}
    \Delta_2 & \le \sum_{s=1}^{r-1} \rms N_s N_{r-s-1}^2 p^{2\rpt - \spt - 2}, \\
    \label{eq:Delta3}
    \Delta_3 & \le \sum_{s=1}^{r-2} \rms N_s N_{r-s-1}^2 p^{2\rpt - \sppt - 1},
  \end{align}
  where $s$ is the number of common vertices that $K_1$ and $K_2$ share outside of the part containing their missing edges. Moreover, note that $\Delta_3 = 0$ if $r < 3$. Similarly,
  \begin{equation}
    \label{eq:Delta4}
    \begin{split}
      \Delta_4 & \le  \sum_{s=2}^{r-2} \rmms N_s N_{r-s-1}^2 p^{2\rpt - \st - 2} \\
      & + 4 \sum_{s=1}^{r-2} \rmms N_s N_{r-s-1} N_{r-s-2} p^{2\rpt - \spt - 2} \\
      & + 4\sum_{s=0}^{r-2} \rmms N_s N_{r-s-2}^2 p^{2\rpt - \sppt - 2},
    \end{split}
  \end{equation}
  where the first, second, and third lines above correspond to the pairs $K_1 \sim K_2$ that share no, one, and two vertices in the two parts of $\Pi$ that contain the missing edges of $K_1$ and $K_2$; similarly as above, $s$ is the number of common vertices that $K_1$ and $K_2$ share outside of the two parts of $\Pi$ that containing the missing edges.

  Since the maximum degree of $U(T)$ is at most $D$, it is now easy to see that
  \[
  \begin{split}
    \Delta & \le \sum_{i=1}^r e(U_i(T))^2 \cdot \Delta_1 + \sum_{v \in [n]} \deg_{U(T)}(v)^2 \cdot \Delta_2 + e(U(T)) \cdot \Delta_3 + \sum_{i \neq j} e(U_i(T))e(U_j(T)) \cdot \Delta_4 \\
    & \le e(U(T))^2 \cdot \max\{\Delta_1, \Delta_4\} + 2De(U(T)) \cdot \Delta_2 + e(U(T)) \cdot \Delta_3.
  \end{split}
  \]

  Recall that $e(T) \le \delta m$ and that
  \begin{equation}
    \label{eq:p-order}
    p = \frac{m - e(T)}{e(\Pi)} \ge \frac{m}{2e(\Pi)} \ge \frac{m}{n^2} \gg n^{-\frac{2}{r+2}}.
  \end{equation}
  It follows from~\eqref{eq:p-order} that $n^sp^{\st} \gg n^2p$ for every $s \in \{3, \ldots, r\}$ and $n^sp^{\st} \gg n^3p^3$ for every $s \in \{4, \ldots, r\}$. Therefore, the sums in the right-hand sides of \eqref{eq:Delta1}, \eqref{eq:Delta2}, and~\eqref{eq:Delta3} are dominated by the terms with $s$ equal to $2$, $1$, and $1$, respectively, and hence
  \begin{align*}
    \Delta_1 & \le 
    (1+o(1)) \binom{r-1}{2}\left(\frac{1}{r} + \gamma\right)^{2r-4}n^{2r-4}p^{2\rpt-3}, \\
    \Delta_2 & \le
    (1+o(1)) \binom{r-1}{1}\left(\frac{1}{r} + \gamma\right)^{2r-3}n^{2r-3}p^{2\rpt-3}, \\
    \Delta_3 & \le
    (1+o(1)) \binom{r-1}{1}\left(\frac{1}{r} + \gamma\right)^{2r-3}n^{2r-3}p^{2\rpt-4}  \cdot \one[r \ge 3].
  \end{align*}
  Similarly, the three sums in the right-hand side of~\eqref{eq:Delta4} are dominated by the terms with $s$ equal to $2$, $1$, and $0$, respectively, and hence
  \[
  \Delta_4 \le (1+o(1)) \left[ \binom{r-2}{2} + 4\binom{r-2}{1} + 4\binom{r-2}{0} \right] \left(\frac{1}{r} + \gamma\right)^{2r-4}n^{2r-4}p^{2\rpt-3}.
  \]
  The (somewhat crude) estimates
  \[
  \max\left\{ \binom{r-1}{2}, \binom{r-2}{2} + 4\binom{r-2}{1} + 4\binom{r-2}{0} \right\} < 2r^2 \quad \text{and} \quad \frac{1}{r} + \gamma < 1
  \]
  yield that for sufficiently large $n$,
  \begin{equation}
    \label{eq:Delta}
    \Delta \le e(U(T)) \cdot n^{2r-4} p^{2\rpt - 4} \cdot \Big( 2 r^2 e(U(T)) p + 2rDnp + \one[r \ge 3] \cdot rn \Big).
  \end{equation}

  \medskip
  \noindent
  \textbf{Choosing the right value for $q$.}
  Recall the definition of $\zeta$ from \eqref{eq:crgammap}. With foresight, we let
  \[
  q = \frac{\zeta r^{r-1} \log m}{n^{r-1} p^{\rpt-1}}.
  \]
  First, let us check that $q \le 1$. Note that by our assumption on $m$ and $T$,
  \[
  m - e(T) \ge (1-\delta)m \ge \left(1 - \frac{\eps}{4(1+\eps)}\right) m \ge \left(1 - \frac{\eps}{4(1+\eps)}\right)(1+\eps)m_r = \left(1+\frac{3\eps}{4}\right)m_r
  \]
  and therefore by~\eqref{eq:pr-mr},
  \[
  p = \frac{m-e(T)}{e(\Pi)} \ge \left(1+\frac{3\eps}{4}\right) \frac{m_r}{\left(1-\frac{1}{r}\right)\frac{n^2}{2}} = \left(1+\frac{3\eps}{4}\right) p_r.
  \]
  It follows that (recalling the definition of $p_r$ from~\eqref{eq:pr-def})
  \[
  \begin{split}
  \frac{n^{r-1}p^{\rpt-1}}{r^{r-1}} & \ge \left(1+\frac{3\eps}{4}\right)^{\rpt-1} \left(\frac{n}{r}\right)^{r-1} p_r^{\rpt-1} = \left(1+\frac{3\eps}{4}\right)^{\rpt-1} \left(2-\frac{2}{r+2}\right) \log n \\
  & \ge \zeta \cdot \left(1+\frac{\eps}{3}\right) \left(2-\frac{2}{r+2}\right) \log n \ge \zeta \cdot \log m,
  \end{split}
  \]
  where the second inequality follows from~\eqref{eq:zeta}, and hence $q \le 1$; to see the last inequality, note that if $m \gg m_r$, then we may assume that $\eps = 1$.

  With the aim of establishing~\eqref{eq:q-goal}, observe first that, by~\eqref{eq:mu}, \eqref{eq:Delta}, and the inequality $\gamma \le \frac{1}{2r}$,
  \begin{equation}
    \label{eq:q-Delta-mu}
    \begin{split}
      \frac{q\Delta}{\mu} & \le (2r^2)^{r-1} \zeta \cdot \log m \cdot \left[ \frac{2r^2e(U(T))}{n^2p} + \frac{2rD}{np} + \frac{\one[r \ge 3] \cdot r}{np^2} \right] \\
      & \le (2r^2)^r \zeta \cdot \log n \cdot \left[ \frac{2Dn}{m} + o\big(n^{-1/5}\big) \right] \le 2\rho,
    \end{split}
  \end{equation}
  where the second inequality follows from the fact that $n^2p \ge m$, see~\eqref{eq:p-order}, and the fact that $e(U(T)) \le Dn$, see~\eqref{eq:m_0}, while the final inequality follows from~\eqref{eq:m_0}. Recall the definitions of $q$ and $\zeta$. It follows from~\eqref{eq:mu} and \eqref{eq:q-Delta-mu} that
  \[
  \begin{split}
    q\mu - q^2\Delta/2 & \ge (1-\rho)q\mu \ge (1-\rho)q \cdot e(U(T)) \cdot \left( \frac{1}{r} - \gamma \right)^{r-1} n^{r-1} p^{\rpt-1} \\
    & = e(U(T)) \cdot (1+\eps) \cdot \log m.
  \end{split}
  \]
  This implies~\eqref{eq:q-goal}, thus completing the proof.
\end{proof}

\subsection{The low degree case}

\label{sec:sparse-low-degree-case}

In this section, we handle the low degree case, i.e., we count all graphs $G$ in $\Fs$ with $T(G) \in \TTL$. Our goal is to prove the following lemma, cf.\ Theorem~\ref{thm:1-statement}.

\begin{lemma}
  \label{lemma:low-degree}
  If $n$ is sufficiently large, then
  \[
  \left|\left\{G \in \Fs \colon T(G) \in \TTL \right\}\right| \le m^{-\eps/4} \cdot \binom{e(\Pi)}{m}.
  \]
\end{lemma}
\begin{proof}
  Let us first further partition the family $\TTL$. For integers $t$ and $u$, let $\TTL_{t,u}$ be the collection of all $T \in \TTL$ with $e(T) = t$ and $e(U(T)) = u$ and let
  \[
  \Iu = \left\{(t^*,x,h) \colon t^* + xD/2 \le u \text{ and } h \le \frac{\eps \xi}{6} \cdot \frac{n \log m}{m} \cdot u \right\}.
  \]
  Observe that $|\Iu| \le u^3$. It follows from the definition of $\TTL$, see~\eqref{eq:low-high-sparse}, and~\eqref{eq:UT-lower} that
  \[
  \TTL_{t,u} \subseteq \bigcup_{(t^*,x,h) \in \Iu} \TT_t(t^*,x,h).
  \]

  By Lemma~\ref{lemma:T-count}, if $n$ is sufficiently large, then
  \begin{equation}
    \label{eq:TTL}
    \begin{split}
      |\TTL_{t,u}| \cdot \binom{e(\Pi)}{m-t} & \le \sum_{(t^*,x,h) \in \Iu} e^{1/\xi} \cdot m^{t^* + xD/2} \cdot \exp\left(\frac{2mh}{\xi n}\right) \cdot \binom{e(\Pi)}{m} \\
      & \le \sum_{(t^*,x,h) \in \Iu} e^{1/\xi} \cdot m^{t^*+xD/2 + \eps u/3} \cdot \binom{e(\Pi)}{m} \\
      & \le u^3 \cdot e^{1/\xi} \cdot m^{(1+\eps/3)u} \cdot \binom{e(\Pi)}{m} \le m^{(1+2\eps/3)u} \cdot \binom{e(\Pi)}{m}.
    \end{split}
  \end{equation}
  Furthermore, since clearly $e(U(T)) \ge \min\{e(T),D\}$ for all $T$, it follows from~\eqref{eq:TTL}, and Lemma~\ref{lemma:Janson} that
  \[
  \begin{split}
    \sum_{T \in \TTL} |\Fst| & \le \sum_{t = 1}^{\delta m} \sum_{u = \min\{t,D\}}^t m^{-(1+\eps)u} \cdot |\TTL_{t,u}| \cdot \binom{e(\Pi)}{m-t} \\
    & \le \sum_{t = 1}^{\delta m} \sum_{u = \min\{t,D\}}^t m^{-\eps u / 3} \cdot \binom{e(\Pi)}{m} \le m^{-\eps/4} \cdot \binom{e(\Pi)}{m},
  \end{split} 
  \]
  where in the last inequality we used the fact that $D \gg 1$ and hence
  \[
  \sum_{t = 1}^{\delta m} \sum_{u = \min\{t,D\}}^t m^{-\eps u / 3} \le \sum_{u = 1}^D m^{-\eps u / 3} + (\delta m)^2 \cdot m^{-\eps D / 3} \ll m^{-\eps / 4}.
  \]
  This completes the proof in the low degree case.  
\end{proof}

\subsection{The high degree case}

\label{sec:sparse-high-degree-case}

Recall the definition of $\TTH$, see~(\ref{eq:low-high-sparse}). In this section, we shall enumerate graphs in the family $\FFH$ defined by
\[
\FFH = \{G \in \Fs \colon T(G) \in \TTH\} = \bigcup_{T \in \TTH} \Fst.
\]
Our goal will be proving the following lemma, which together with Lemma~\ref{lemma:low-degree} readily implies Theorem~\ref{thm:1-statement}.

\begin{lemma}
  \label{lemma:high-degree}
  If $n$ is sufficiently large, then
  \[
  \left|\FFH\right| = \left|\left\{G \in \Fs \colon T(G) \not\in \TTL \right\}\right| \le 3\exp\left(-\frac{m}{n}\right) \cdot \binom{e(\Pi)}{m}.
  \]
\end{lemma}

We start with the following observation. Fix a $T \in \TTH$ and suppose that for some $i \in [r]$, a vertex $v \in V_i$ satisfies $\deg_T(v) \ge \beta m/n$. Since every graph $G \in \Fst$ satisfies 
\[
\deg_G(v,V_j) \ge \deg_G(v,V_i) = \deg_T(v) \ge \beta m/n \quad \text{for every $j \in [r]$},
\]
see~\eqref{eq:Pi-unfriendly}, and is $K_{r+1}$-free, no matter how we choose the edges of $G$ that are incident to $v$, there will be at least $(\beta m / n)^r$ copies of $K_r$ (those induced by one vertex from each $N_G(v) \cap V_j$ with $j \in [r]$) that cannot be fully contained in the graph $G \cap \Pi$.

Given an arbitrary vertex $v$, assuming that its neighbors in $G$ have already been chosen, let $\HH_v^*$ be the collection of all $\prod_{j=1}^r \deg_G(v,V_j)$ such forbidden copies of~$K_r$, that is, let
\[
\HH_v^* = (N_G(v) \cap V_1) \times \ldots \times (N_G(v) \cap V_r).
\]
We furthermore let
\begin{equation}
  \label{eq:Ds}
  \Ds = \frac{\beta m}{2n}.
\end{equation}
Recall the definitions from Section~\ref{sec:counting-graphs}. Fix some $t$, $t^*$, $x$, and $h$, pick an arbitrary $T \in \TT_t'(t^*,x,h)$, and let
\[
b = \left\lceil \frac{h}{2r} \right\rceil.
\]
Let us stress the fact that we select $T$ from $\TT_t'(t^*,x,h)$ and not from $\TT_t(t^*,x,h)$, which means that $T$ contains exactly (and not at most) $h$ vertices with degree exceeding $\beta m / n$.

\begin{claim}
  \label{claim:H'}
  There is an $i \in [r]$ and a set $H' \subseteq H(T) \cap V_i$ of $b$ vertices such that
  \begin{equation}
    \label{eq:H'}
    \deg_T(v, V_i \setminus H') \ge \Ds \quad \text{for every $v \in H'$}.
  \end{equation}
\end{claim}
\begin{proof}
  Since $T$ has $h$ vertices with degree at least $\beta m / n$, some $V_i$ contains at least $h / r$ of them. This set $V_i$ can be partitioned into two sets $V_i'$ and $V_i''$ in such a way that $\deg_T(v, V_i'') \ge \deg_T(v, V_i')$ for each $v \in V_i'$ and, vice versa, $\deg_T(v, V_i') \ge \deg_T(v, V_i'')$ for each $v \in V_i''$. For example, one may consider a maximum cut in $T[V_i]$. One of these two parts, $V_i'$ or $V_i''$, contains at least $h / 2r$ vertices with degree at least $\beta m / n$ in $T$. We let $H'$ be an arbitrary $b$-element subset of such a set. It is easily checked that $H'$ satisfies~\eqref{eq:H'}.
\end{proof}

For every $T \in \TT_t'(t^*,x,h)$, we choose some arbitrary set $H'$ as in Claim~\ref{claim:H'}. Next, given a graph $G \in \Fst$, for every $v \in H'$ and each $j \in [r]$, let $W_j(v)$ be a canonically chosen $\Ds$-element subset of $N_G(v) \cap (V_j \setminus H')$. Given such $G$, consider the $r$-uniform hypergraph $\HH'$ defined by
\[
\HH' = \bigcup_{v \in H'} W_1(v) \times \ldots \times W_r(v).
\]
Note that $\HH' \subseteq \bigcup_{v \in H'} \HH_v^*$, that is, every edge ($r$-tuple) in $\HH'$ represents a copy of $K_r$ that is forbidden to appear in $G$. We will enumerate graphs in $\FFH$ using two different methods, depending on the number and the distribution of edges in the hypergraph $\HH'= \HH'(G)$. Before we make this precise, we need a few more definitions.

Given an arbitrary $\HH \subseteq V_1 \times \ldots \times V_r$, an $I \subseteq [r]$, and an $L \in \prod_{j \in I} V_j$, we define the degree of $L$ in $\HH$, denoted $\deg_\HH(L)$, by
\[
\deg_\HH(L) = |\{ K \in \HH \colon L \subseteq K \}|.
\]
For $s \in [r]$, the maximum $s$-degree $\Delta_s(\HH)$ of $\HH$ is defined by
\[
\Delta_s(\HH) = \max \Big\{ \deg_\HH(L) \colon L \in \prod_{j \in I} V_j \text{ for some $I \subseteq [r]$ with $|I| = s$}\Big\}.
\]
We will measure the uniformity of the distribution of the edges of $\HH$ in terms of these maximum $s$-degrees. First, let us fix several additional parameters. Let $C_1$ be a constant satisfying
\begin{equation}
  \label{eq:C1}
  \frac{3\beta C_1}{2r} \ge \frac{6}{\eps \xi} + \frac{4}{\xi} + 3.
\end{equation}
Next, let
\begin{equation}
  \label{eq:lambda-alpha}
  \lambda = \frac{1}{2^{r+1}} \qquad \text{and} \qquad \alpha = \exp(-6C_1-1)
\end{equation}
and let $\tau$ be a small positive constant such that Lemma~\ref{lemma:d-sets} holds with $\tau$ and with $\alpha$, $\lambda$, and each $k \in \{2, \ldots, r\}$, i.e.,
\begin{equation}
  \label{eq:tau}
  \tau = \min_{2 \le k \le r} \tau_{\ref{lemma:d-sets}}(k, \alpha, \lambda).
\end{equation}
Finally,
\begin{equation}
  \label{eq:C2-sigma}
  C_2 = \frac{(2r)^r}{\tau} \quad \text{and} \quad \sigma = \frac{\tau}{(2r)^r}.
\end{equation}

We are finally ready to partition the family $\FFH$ into the regular and irregular cases, according to the edge distribution of the hypergraphs $\HH'$. First, we let $\FFR_1$ be the family of all $G \in \FFH$ such that $e(\HH') \ge \sigma n^r$. Second, we let
\begin{equation}
  \label{eq:c2}
  c_2 = \frac{\beta^r}{2^{r+3}r}
\end{equation}
and define $\FFR_2$ to be the family of all $G \in \FFH \setminus \FFR_1$ such that $\HH'$ contains a subhypergraph $\HH$ satisfying
\begin{equation}
  \label{eq:FF2-eHH}
  e(\HH) \ge c_2 \cdot |H(T(G))| \cdot \left(\frac{m}{n}\right)^r = |H(T(G))| \cdot \frac{(\Ds)^r}{8r}
\end{equation}
and
\begin{equation}
  \label{eq:FF2-DeltaHH}
  \Delta_s(\HH) \le \max \left\{  \left(\frac{m}{n}\right)^{r-s} , C_2 \cdot \frac{e(\HH)}{n^s} \right\} \quad \text{for every $s \in \{2, \ldots, r-1\}$}.
\end{equation}
Finally, we let $\FFI = \FFH \setminus (\FFR_1 \cup \FFR_2)$. Counting of graphs in $\FFR_1 \cup \FFR_2$ and $\FFI$ will be referred to as the regular and irregular cases, respectively. In the next two sections, we will prove the following estimates, which readily imply Lemma~\ref{lemma:high-degree}.

\begin{lemma}
  \label{lemma:regular-case}
  If $n$ is sufficiently large, then
  \[
  |\FFR_1| \le \exp\left(-\frac{\sigma m}{2^{r+3}}\right) \cdot \binom{e(\Pi)}{m} \qquad \text{and} \qquad |\FFR_2| \le \exp\left(-\frac{m}{n}\right) \cdot \binom{e(\Pi)}{m}.
  \]
\end{lemma}

\begin{lemma}
  \label{lemma:irregular-case}
  If $n$ is sufficiently large, then
  \[
  |\FFI| \le \exp\left(-\frac{m}{n}\right) \cdot \binom{e(\Pi)}{m}.
  \]
\end{lemma}

\subsection{The regular case}

\label{sec:regular-case}

In this section, we bound the number of graphs that fall into the regular case, that is, we prove Lemma~\ref{lemma:regular-case}. Our main tool will be the following two lemmas that provide upper bounds on the number of subgraphs of $\Pi$ that do not fully contain any member of a collection of forbidden copies of $K_r$ which is either very large (Lemma~\ref{lemma:good-H-basic}) or whose members are somewhat uniformly distributed (Lemma~\ref{lemma:good-H}). The proof of both of these lemmas is another application of the Hypergeometric Janson Inequality (Lemma~\ref{lemma:HJI}).

\begin{lemma}
  \label{lemma:good-H-basic}
  Suppose that $\HH \subseteq V_1 \times \ldots \times V_r$ satisfies $e(\HH) \ge \sigma n^r$. Then for every $m'$ with $m/2 \le m' \le m$, the number of subgraphs of $\Pi$ with $m'$ edges that do not fully contain a copy of $K_r$ whose vertex set is an edge of $\HH$ is at most
  \[
  2 \cdot \exp\left( - \frac{\sigma}{2^{r+1}} \cdot m\right) \cdot \binom{e(\Pi)}{m'}.
  \]
\end{lemma}
  
\begin{lemma}
  \label{lemma:good-H}
  There exists a positive $c$ such that the following holds. Suppose that $\HH \subseteq V_1 \times \ldots \times V_r$ satisfies $e(\HH) \ge B(m/n)^r$ for some $B$ and~\eqref{eq:FF2-DeltaHH} holds, that is,
  \[
  \Delta_s(\HH) \le \max\left\{\left(\frac{m}{n}\right)^{r-s}, C_2 \cdot \frac{e(\HH)}{n^s}\right\} \quad \text{for every $s \in \{2, \ldots, r-1\}$}.
  \]
  Then, for every $m'$ with $m/2 \le m' \le m$, the number of subgraphs of $\Pi$ with $m'$ edges that do not fully contain a copy of $K_r$ whose vertex set is an edge of $\HH$ is at most
  \[
  2 \cdot \exp\left(-\min\left\{\frac{B\log n}{n}, 1\right\} \cdot c m\right) \cdot \binom{e(\Pi)}{m'}.
  \]
\end{lemma}
\begin{proof}[Proof of Lemmas~\ref{lemma:good-H-basic} and~\ref{lemma:good-H}]
  We use the Hypergeometric Janson Inequality to count graphs satisfying our constraint. Denote the number of them by $N$. Let $\KK$ be the collection of (the edge sets of) all copies of $K_r$ whose vertex set belongs to $\HH$, let $p = \frac{m'}{e(\Pi)}$, and let
  \[
  \mu = \sum_{K \in \KK} p^{|K|} \qquad \text{and} \qquad \Delta = \sum_{K_1 \sim K_2} p^{|K_1 \cup K_2|},
  \]
  where the second sum above is over all ordered pairs $(K_1, K_2) \in \KK^2$ such that $K_1$ and $K_2$ share at least one edge but $K_1 \neq K_2$. By the Hypergeometric Janson Inequality, letting $q = \min\{1, \frac{\mu}{\Delta}\}$, we have
  \[
  N \le 2 \cdot \exp\left(-\min\left\{\frac{\mu}{2}, \frac{\mu^2}{2\Delta}\right\}\right) \cdot \binom{e(\Pi)}{m'}.
  \]
  Hence, in order to establish Lemma~\ref{lemma:good-H-basic}, it suffices to show that if $e(\HH) \ge \sigma n^r$, then
  \begin{equation}
    \label{eq:good-H-basic-goal}
    \min\left\{\mu, \frac{\mu^2}{\Delta}\right\} \ge \frac{\sigma}{2^r} \cdot m
  \end{equation}
  and in order to establish Lemma~\ref{lemma:good-H}, it suffices to show that under appropriate assumptions on $\HH$, there exists a positive $c$ that depends only on $r$ and $C_2$ such that
  \begin{equation}
    \label{eq:good-H-goal}
    \min\left\{\mu, \frac{\mu^2}{\Delta}\right\} \ge 2 \cdot \min\left\{ \frac{B\log n}{n}, 1 \right\} \cdot c m.
  \end{equation}

  To this end, observe that
  \[
  \mu = e(\HH) \cdot p^{\rt} \qquad \text{and} \qquad \Delta \le e(\HH) \cdot \sum_{s=2}^{r-1} \binom{r}{s} \Delta_s(\HH) p^{2\rt - \st},
  \]
  where the $s$th term of the sum in the upper bound on $\Delta$ estimates the contribution of pairs $K_1 \sim K_2$ with $|V(K_1) \cap V(K_2)| = s$. It follows from our assumptions, see~\eqref{eq:ePi-lower}, that
  \begin{equation}
    \label{eq:p-mr-n}
    \frac{8m}{n^2} \ge \frac{m}{e(\Pi)} \ge p = \frac{m'}{e(\Pi)} \ge \frac{m}{2e(\Pi)} \ge \frac{m_r}{2e(\Pi)} \ge \frac{m_r}{\left(1-\frac{1}{r}\right)n^2} \gg n^{-\frac{2}{r+2}}    
  \end{equation}
  and hence for every $s \in \{2, \ldots, r\}$,
  \begin{equation}
    \label{eq:good-H-basic-np}
    n^s p^{\st} \ge n^2 p = n^2 \cdot \frac{m'}{e(\Pi)} \ge n^2 \cdot \frac{m}{2e(\Pi)} \ge m,
  \end{equation}
  which implies, in particular, that under the assumptions of Lemma~\ref{lemma:good-H-basic}, we have $\mu \ge \sigma m$. Moreover, it follows from \eqref{eq:p-mr-n} that for every $s \in \{2, \ldots, r\}$, recalling \eqref{eq:pr-mr},
  \begin{equation}
    \label{eq:nspspt}
    \begin{split}
      \left(\frac{m}{n}\right)^{s-1} p^{\st} & \ge \left(\frac{n}{8}\right)^{s-1} p^{\st+s-1} = \left(\frac{n}{8}\right)^{s-1} p^{\spt - 1} \ge \left(\frac{n}{8}\right)^{s-1} \cdot \left(\frac{m_r}{\left(1-\frac{1}{r}\right)n^2}\right)^{\spt-1} \\
      & = \left(\frac{n}{8}\right)^{s-1} \cdot \left(\frac{p_r}{2}\right)^{\spt-1} \ge \frac{r^{r-1}}{4^{r^2}} \cdot \log n.
    \end{split}
  \end{equation}
  To see the last inequality, note that if $s = r$, then it follows immediately from~\eqref{eq:pr-def}. On the other hand, if $2 \le s < r$, then actually
  \[
  n^{s-1} p_r^{\spt - 1} \gg n^{s-1} \left(n^{-\frac{2}{r+2}}\right)^{\spt - 1} = n^{(s-1)\left(1 - \frac{s+2}{r+2}\right)} \gg \log n.
  \]
  One now easily deduces from \eqref{eq:nspspt} that under the assumptions of Lemma~\ref{lemma:good-H},
  \begin{equation}
    \label{eq:good-H-mu}
    \mu \ge B \left(\frac{m}{n}\right)^{r} p^{\rt} \ge \frac{B \log n}{n} \cdot \frac{r^{r-1}}{4^{r^2}} \cdot m.
  \end{equation}

  We now turn to estimating $\mu^2/\Delta$. In the context of Lemma~\ref{lemma:good-H-basic}, we simply use the trivial bound $\Delta_s(\HH) \le n^{r-s}$ and deduce that for each $s \in \{2, \ldots, r-1\}$,
  \[
  \Delta_s(\HH) p^{-\st} \le n^{r-s} p^{-\st} \le \frac{n^r}{m},
  \]
  where the last inequality follows from~\eqref{eq:good-H-basic-np}.  It follows that
  \[
  \Delta \le 2^r \cdot p^{2 \rt} \cdot \frac{n^r}{m} \cdot e(\HH)
  \]
  and hence
  \[
  \frac{\mu^2}{\Delta} \ge \frac{1}{2^r} \cdot e(\HH) \cdot \frac{m}{n^r} \ge \frac{\sigma}{2^r} \cdot m,
  \]
  which implies~\eqref{eq:good-H-basic-goal}, as we have already seen that $\mu \ge \sigma m$. In the context of Lemma~\ref{lemma:good-H}, it follows from~\eqref{eq:good-H-basic-np} and~\eqref{eq:nspspt} that for each $s \in \{2, \ldots, r-1\}$,
  \[
  \Delta_s(\HH)p^{-\st} \le \max\left\{ \frac{(m/n)^{r-1}}{(m/n)^{s-1}p^{\st}}, C_2 \cdot \frac{e(\HH)}{n^sp^{\st}} \right\} \le \max\left\{ \frac{4^{r^2}}{r^{r-1}} \cdot \frac{(m/n)^{r-1}}{\log n}, C_2 \cdot \frac{e(\HH)}{m} \right\}
  \]
  and therefore,
  \[
  \frac{\mu^2}{\Delta} \ge \frac{1}{2^r} \cdot \min\left\{ \frac{r^{r-1}}{4^{r^2}} \cdot \frac{e(\HH) \log n}{(m/n)^{r-1}}, \frac{m}{C_2} \right\} \ge \min\left\{ \frac{B \log n}{n}, 1 \right\} \cdot \min\left\{\frac{r^{r-1}}{4^{r^2+r}}, \frac{1}{2^r C_2} \right\} \cdot m,
  \]
  which, together with~\eqref{eq:good-H-mu}, implies~\eqref{eq:good-H-goal}, completing the proof.
\end{proof}

\begin{proof}[{Proof of Lemma~\ref{lemma:regular-case}}]
  Recall the definitions of $\FFR_1$ and $\FFR_2$ from Section~\ref{sec:sparse-high-degree-case}. We first show that the family $\FFR_1$ is small. In order to construct a graph $G \in \FFR_1$, we first choose $h$ and $t$ and restrict our attention to graphs $G$ satisfying $t = e(T(G))$ and $h = |H(T(G))|$. Clearly for each such $G$,
  \begin{equation}
    \label{eq:h-bound}
    h \cdot \frac{\beta m}{n} \le \sum_v \deg_{T(G)}(v) = 2t \le 2\delta m.
  \end{equation}
  Then, we choose the set $H'$ of $b$ vertices from some $V_i$, see Claim~\ref{claim:H'}, and for each $v \in H'$, we choose the sets $W_1(v), \ldots, W_r(v)$ of size $\Ds$ each. After these are fixed, we choose the remaining $t' = t - b\Ds$ edges of $T(G)$ and the remaining $m - t' - br\Ds$ (that is, $m - t - b(r-1)\Ds$) edges of $G \cap \Pi$ in such a way that $G \cap \Pi$ contains no copy of $K_r$ whose vertex set is an edge of the hypergraph $\HH'$ (defined in the previous section). The main point is that the assumption that $G \in \FFR_1$ means that $e(\HH') \ge \sigma n^r$ and hence we may use Lemma~\ref{lemma:good-H-basic} to bound the number of choices for $G \cap \Pi$.

  The number $Z_1$ of ways to choose the sets $W_1(v) \subseteq V_1, \ldots, W_r(v) \subseteq V_r$ for each $v$ satisfies
  \begin{equation}
    \label{eq:Z1}
    Z_1 \le \prod_{j = 1}^r \binom{|V_j|}{\Ds} \le \binom{|V_1| + \ldots + |V_r|}{r \cdot \Ds} = \binom{n}{r\Ds}.
  \end{equation}
  It now follows from the definition of $\FFR_1$ and Lemma~\ref{lemma:good-H-basic} that
  \[
  |\FFR_1| \le \sum_{t, h} \binom{n}{b} \cdot \binom{n}{r\Ds}^b \cdot |\TT_{t'}| \cdot 2 \cdot \exp\left(-\frac{\sigma}{2^{r+1}} \cdot m\right) \cdot \binom{e(\Pi)}{m-t'-br\Ds}.
  \]
  A computation along the lines of the proof of Lemmas~\ref{lemma:T-count-basic} and~\ref{lemma:T-count}, see~\eqref{eq:T-count-step-ratio} and~\eqref{eq:T-count-step-high-deg}, shows that
  \begin{equation}
    \label{eq:comp-T-count}
    \begin{split}
      2\binom{n}{b} \binom{n}{r\Ds}^b \binom{e(\Pi)}{m-t'-rb\Ds} & \le 2 \binom{n}{b} \left(\frac{em}{\xi n r \Ds}\right)^{r \Ds b} \cdot \binom{e(\Pi)}{m-t'} \\
      & \le \exp\left(\frac{2mb}{\xi n}\right) \cdot \binom{e(\Pi)}{m-t'}.
    \end{split}
  \end{equation}
  To see the last inequality, recall that the value of the function $x \mapsto (a/x)^x$ is maximized when $x = a/e$, which implies that $(\frac{em}{\xi n r \Ds})^{r\Ds b} \le \exp(\frac{mb}{\xi n})$.
  Hence, by Lemma~\ref{lemma:T-count-basic}, since $b \le h$,
  \begin{equation}
    \label{eq:FF1-second}
    |\FFR_1| \le \sum_{t, h} \exp\left(\frac{2mh}{\xi n} - \frac{\sigma m}{2^{r+1}}\right) \cdot \left(\frac{e}{\xi \delta}\right)^{\delta m} \cdot \binom{e(\Pi)}{m}.
  \end{equation}
  Since
  \[
  \left(\frac{e}{\xi \delta}\right)^\delta \le \exp\left(\frac{\sigma}{2^{r+2}}\right) \qquad \text{and} \qquad \frac{4\delta}{\beta\xi} \le \frac{\sigma}{2^{r+4}}
  \]
  for sufficiently small $\delta$, continuing~\eqref{eq:FF1-second}, we have, by~\eqref{eq:h-bound},
  \[
  |\FFR_1| \le m^2 \exp\left(\frac{4\delta m}{\beta \xi} - \frac{\sigma m}{2^{r+1}} + \frac{\sigma m}{2^{r+2}} \right) \cdot \binom{e(\Pi)}{m} \le \exp\left(-\frac{\sigma m}{2^{r+3}}\right) \cdot \binom{e(\Pi)}{m}.
  \]

  In order to complete the proof, we still need to show that the family $\FFR_2$ is also small. We count the graphs in $\FFR_2$ almost the same way as we counted the graphs in $\FFR_1$. That is, we first choose $h$ and $t$ and consider only graphs $G$ with $t = e(T(G))$ and $h = |H(T(G))|$. Then, with $t$ and $h$ fixed, we choose the set $H'$ of $b$ vertices from some $V_i$ and for each $v \in H'$, we select the sets $W_1(v), \ldots, W_r(v)$. Finally, we choose the remaining $t' = t - b\Ds$ edges of $T(G)$ and the remaining $m - t' - br\Ds$ edges of $G \cap \Pi$. The assumption that $G \in \FFR_2$ means that we may use Lemma~\ref{lemma:good-H} with $B = c_2h$ to bound the number of choices for $G \cap \Pi$.

  The main difference in our treatment of $\FFR_2$, compared to the argument for $\FFR_1$ given above, is that we now use a stronger bound on the number of choices of the $t'$ edges of $T(G)$ that are selected in the second stage of the above procedure. To this end, we fix some $x$ and $t^*$ and further restrict our attention to graphs $G$ that satisfy $x = |X(T(G))|$ and $t^* = e(T(G) - X(T(G)))$. In particular, we are only counting graphs $G \in \FFR_2$ that satisfy $T(G) \in \TT'_t(t^*, x, h)$. Let $T'$ be the graph consisting of the $t'$ edges of $T(G)$ that we choose after selecting the $b\Ds$ edges of $T(G)$ that we fixed while we were choosing $W_i(v)$ for all $v \in H'$. Since $T(G) \in \TT'_t(t^*,x,h)$ and all edges of $T(G) \setminus T'$ have an endpoint in $H'$, it is not hard to see that $T'$ must be in $\TT_{t'}(t^*, x, h)$. Indeed, the set $H(T(G))$ contains all vertices whose degree in $T(G)$ exceeds $\beta m / n$ (and hence also all vertices whose degree in $T'$ exceeds $\beta m / n$), and we may obtain $T'$ from $T(G)$ by deleting only edges incident to $H' \subseteq H(T(G)) \subseteq X(T(G))$, which means that the number of edges that have no endpoints in the set $X(T(G))$ is $t^*$ in both $T(G)$ and $T'$. With this additional information about $T'$, we may now appeal to Lemma~\ref{lemma:T-count} in place of Lemma~\ref{lemma:T-count-basic} in order to get a stronger bound on the number of choices for $T'$. It now follows (cf.~the calculation leading up to~\eqref{eq:FF1-second}) from the definition of $\FFR_2$ and Lemma~\ref{lemma:good-H} with $B = c_2 h$ that, letting $c$ be the constant from the statement of Lemma~\ref{lemma:good-H},
  \begin{equation}
    \label{eq:FF2}
    |\FFR_2| \le \sum_{t,t^*,x,h} \exp\left( \frac{2mb}{\xi n} \right) \cdot |\TT_{t'}(t^*, x, h)| \cdot \exp\left( - \min\left\{\frac{c_2 h \log n}{n}, 1\right\} \cdot c m\right) \cdot \binom{e(\Pi)}{m-t'}.
  \end{equation}
  Let $F_2(t,t^*,x,h)$ denote the term in the sum in the right hand side of~\eqref{eq:FF2}. If $c_2 h \log n \ge n$, then we use the fact that $\TT_{t'}(t^*,x,h) \subseteq \TT_{t'}$ and $t' \le t \le \delta m$ and, using Lemma~\ref{lemma:T-count-basic}, we further estimate $F_2(t,t^*,x,h)$ as follows (recall that $b \le h$):
  \[
  \begin{split}
    F_2(t,t^*,x,h) & \le  \exp\left( \frac{2mh}{\xi n} - c m \right) \cdot \left(\frac{e}{\xi \delta}\right)^{\delta m} \cdot \binom{e(\Pi)}{m} \\
    & \le \exp\left( \frac{4\delta m}{\beta \xi} - c m \right) \cdot \left(\frac{e}{\xi \delta}\right)^{\delta m} \cdot \binom{e(\Pi)}{m} \le \exp\left( -\frac{c m}{4} \right) \cdot \binom{e(\Pi)}{m},
  \end{split}
  \]
  where we used \eqref{eq:h-bound} and the fact that
  \[
  \left( \frac{e}{\xi\delta} \right)^\delta < e^{\frac{c}{2}}  \qquad \text{and} \qquad \frac{4\delta}{\beta\xi} < \frac{c}{4},
  \]
  provided that $\delta$ is sufficiently small. Now, we recall from the definition of $\TTH$, see~\eqref{eq:UT-lower} and \eqref{eq:low-high-sparse}, that since $T(G) \in \TTH$ and we have $x = |X(T(G))|$ and $t^* = e(T(G) - X(T(G)))$, then
  \begin{equation}
    \label{eq:h}
    h = |H(T(G))| > \frac{\eps \xi}{6} \cdot \frac{n \log m}{m} \cdot (t^* + xD/2).
  \end{equation}
  Hence, if $c_2 h \log n < n$, then by Lemma~\ref{lemma:T-count},
  \[
  \begin{split}
    F_2(t,t^*,x,h) & \le e^{1/\xi} \cdot m^{t^* + xD/2} \cdot \exp\left( \frac{m(b+h)}{n} \left( \frac{2}{\xi} - c_2c \log n \right) \right) \cdot \binom{e(\Pi)}{m} \\
    & \le e^{1/\xi} \cdot \exp\left( \frac{mh}{n} \cdot \left(\frac{4}{\xi} + \frac{6}{\eps \xi} - c_2 c \log n\right) \right) \cdot \binom{e(\Pi)}{m} \le \exp\left(- \frac{2m}{n}\right) \cdot \binom{e(\Pi)}{m},
  \end{split}
  \]
  where in the second inequality, we used~\eqref{eq:h} and the fact that $b \le h$, and in the last inequality, we used the facts that $h \ge 1$, which follows from \eqref{eq:h} as $h$ is an integer, and that $n$ is sufficiently large. It follows that
  \[
  |\FFR_2| \le m^2n^2 \cdot \max\left\{ \exp\left(- \frac{cm}{4}\right) , \exp\left(-\frac{2m}{n}\right) \right\} \cdot \binom{e(\Pi)}{m} \le \exp\left( - \frac{m}{n} \right) \cdot \binom{e(\Pi)}{m},
  \]
  provided that $n$ is sufficiently large. This completes the proof in the regular case.
\end{proof}

\subsection{The irregular case}

\label{sec:irregular-case}

In this section, we prove Lemma~\ref{lemma:irregular-case}. In other words, we count those graphs in $\FFH$ for which the hypergraph $\HH'$ of forbidden copies of $K_r$ defined in Section~\ref{sec:sparse-high-degree-case} contains fewer than $\sigma n^r$ edges and does not contain any subhypergraph $\HH$ that satisfies \eqref{eq:FF2-eHH} and \eqref{eq:FF2-DeltaHH}. This is the core of the proof of Theorem~\ref{thm:1-statement}, which makes this section the key section of the paper.

We will describe a procedure that, given a $G \in \FFH \setminus \FFR_1$, constructs some canonical hypergraph $\HH \subseteq \HH'$ by examining the vertices in $H'(T(G))$ and their neighborhoods one by one. By not adding certain $r$-tuples of $\HH'$ to the constructed hypergraph, our procedure forces $\HH$ to satisfy the maximum degree constraints given in~\eqref{eq:FF2-DeltaHH}. For a vast majority of graphs $G \in \FFH \setminus \FFR_1$, the hypergraph $\HH$ will have many edges (i.e., it will satisfy~\eqref{eq:FF2-eHH}), implying that $G \in \FFR_2$. The procedure fails to output a hypergraph with many edges only when the intersections of the neighborhoods of different vertices in $H'(T(G))$ are very far from random-like. Using Lemma~\ref{lemma:d-sets}, we will obtain a bound on the number of graphs with such an atypical distribution of neighborhoods of the vertices in $H'(T(G))$. Since by definition, our procedure has to fail on every graph in $\FFI$, the obtained bound is also an upper bound on $|\FFI|$.

\begin{proof}[{Proof of Lemma~\ref{lemma:irregular-case}}]
  Fix some graph $G \in \FFH \setminus \FFR_1$ and recall the definitions of $\Ds$, $H'(T(G))$, and $W_j(v)$ from Section~\ref{sec:sparse-high-degree-case}. Suppose that $H' = H'(T(G)) = \{v_1, \ldots, v_b\}$, where $v_1 < \ldots < v_b$ (we assume that the vertex set of $G$ is labeled with $\{1, \ldots, n\}$). We now describe the aforementioned procedure which constructs a hypergraph $\HH \subseteq \HH'$.

  \begin{constHH}
    Let $\HH_0 \subseteq V_1 \times \ldots \times V_r$ be the empty hypergraph. For every $\ell = 1, \ldots, b$, do the following:
    \begin{enumerate}[(1)]
    \item
      For every $j \in [r]$, let $W_j = W_j(v_\ell)$.
    \item
      For every $I \subseteq [r]$ with $2 \le |I| \le r-1$, let
      \[
      M_I = \left\{ T \in \prod_{j \in I} V_j \colon \deg_{\HH_{\ell-1}}(T) > \frac{C_2}{2} \cdot \frac{e(\HH_{\ell - 1})}{n^{|I|}} \right\}
      \]
      and let $M_{[r]} = \HH_{\ell - 1}$.
    \item
      Let
      \[
      \HH_\ell = \HH_{\ell - 1} \cup \left\{K \in W_1 \times \ldots \times W_r \colon K \nsupseteq T \text{ for all $T \in \bigcup_I M_I$}\right\}.
      \]
    \end{enumerate}
    Finally, let $\HH = \HH_b$.
  \end{constHH}

  Since $|W_j(v)| = \Ds$ for every $j \in [r]$ and $v \in H'$, then for every $s \in \{2, \ldots, r-1\}$,
  \[
  \Delta_s(\HH) \le \frac{C_2}{2} \cdot \frac{e(\HH)}{n^s} + (\Ds)^{r-s} \le \max\left\{C_2 \cdot \frac{e(\HH)}{n^s}, 2(\Ds)^{r-s}\right\} \le \max\left\{C_2 \cdot \frac{e(\HH)}{n^s}, \left(\frac{m}{n}\right)^{r-s}\right\},
  \]
  that is, $\HH$ satisfies~\eqref{eq:FF2-DeltaHH}. Recall from~\eqref{eq:lambda-alpha} that $\lambda = 2^{-r-1}$. We say that the vertex $v_\ell$ is \emph{useful} if in the $\ell$th iteration of the above algorithm, we have
  \[
  \Big| M_I \cap \prod_{j \in I} W_j \Big| \le \lambda (\Ds)^{|I|} \quad \text{for all $I$ with $2 \le |I| \le r$}.
  \]
  Note that if $v_\ell$ is useful, then
  \[
  e(\HH_\ell) - e(\HH_{\ell - 1}) \ge (\Ds)^r - \sum_{I \subseteq [r]} \lambda (\Ds)^{|I|} \cdot (\Ds)^{r - |I|} = (1-2^r\lambda)(\Ds)^r = \frac{(\Ds)^r}{2}.
  \]
  Therefore, if at least half of the vertices of $H'$ are useful, then (recall the definitions of $\Ds$ and $c_2$ from~\eqref{eq:Ds} and~\eqref{eq:c2})
  \[
  e(\HH) \ge \frac{b}{2} \cdot \frac{(\Ds)^r}{2} \ge \frac{h(\Ds)^r}{8r} = \frac{\beta^rh}{2^{r+3}r} \cdot \left(\frac{m}{n}\right)^r = c_2 h \left(\frac{m}{n}\right)^r,
  \]
  that is, $\HH$ satisfies~\eqref{eq:FF2-eHH}. Hence, if for some $G \in \FFH \setminus \FFR_1$, the above procedure encounters at least $b/2$ useful vertices, then $G \in \FFR_2$. This implies that for every graph $G \in \FFI$, the above procedure encounters fewer than $b/2$ useful vertices. We now enumerate all graphs with this property.

  As before, we first fix $h$ and $t$ and consider only graphs $G$ with $t = e(T(G))$ and $h = |H(T(G))|$. We then choose the $b$ vertices that form the set $H'$ and specify in advance which (at least $b/2$) of them our procedure will mark as not useful. Next, we choose the sets $W_1(v), \ldots, W_r(v)$ in turn for every $v \in H'$, from the one with the smallest label to the one with the largest label (as in the procedure constructing $\HH$). The main point is that for the vertices $v$ that are not useful, we choose the sets $W_i(v)$ in such a way that our procedure will deem them not useful and this severely limits the number of choices for these sets. This is the only stage of the enumeration where we provide a nontrivial upper bound. Finally, we choose the remaining $t' = t - b\Ds$ edges of $T(G)$ and the remaining $m - t' - br\Ds$ edges of $G \cap \Pi$.

  Let us elaborate on the only non-trivial stage of the enumeration described above. We choose the sets $W_1(v), \ldots, W_r(v)$ for vertices $v \in H'$ one by one, following the same order as in the procedure constructing $\HH$. More precisely, suppose that $H' = \{v_1, \ldots, v_b\}$, where $v_1 < \ldots < v_b$, fix some $\ell \in [b]$, and assume that $W_j(v_k)$ have already been chosen for all $j \in [r]$ and $k \in [\ell - 1]$. This means, in particular, that the hypergraph $\HH_{\ell-1}$ and the sets $M_I$ in the $\ell$th iteration of our procedure are already determined for all $I \subseteq [r]$. Clearly, there are at most $\binom{n}{r\Ds}$ ways to choose $W_1(v_\ell) \subseteq V_1, \ldots, W_r(v_\ell) \subseteq V_r$ if $v_\ell$ is useful, see~\eqref{eq:Z1}. Let us now estimate the number of ways to choose these sets in such a way that $v_\ell$ will not be useful, that is, letting $W_j = W_j(v_\ell)$, so that
  \begin{equation}
    \label{eq:MI-not-useful}
    |M_I \cap \prod_{j \in I} W_j| > \lambda d^{|I|} \quad \text{for some $I \subseteq [r]$ with $2 \le |I| \le r$}.
  \end{equation}
  Recall that $\Pi \in \Part(\gamma)$ and hence $|V_j| \ge \frac{n}{2r}$. It follows from the definition of $M_I$ that for every $I \subseteq [r]$ with $2 \le |I| \le r-1$, we have
  \begin{equation}
    \label{eq:MI}
    |M_I| < \frac{2n^{|I|}}{C_2} \le \frac{2^{|I|+1}r^{|I|}}{C_2} \cdot \prod_{j \in I} |V_j| \le \frac{(2r)^r}{C_2} \cdot \prod_{j \in I} |V_j| = \tau \cdot \prod_{j \in I} |V_j|,
  \end{equation}
  where the last inequality follows from~\eqref{eq:C2-sigma}. Since $G \notin \FFR_1$, we also have
  \begin{equation}
    \label{eq:Mr}
    |M_{[r]}| = e(\HH_{\ell - 1}) \le e(\HH') < \sigma n^r \le \sigma \cdot (2r)^r  \cdot \prod_{j=1}^r |V_j| \le \tau \cdot \prod_{j=1}^r |V_j|.
  \end{equation}
  Recalling the definition of $\lambda$ from~\eqref{eq:lambda-alpha}, inequalities~\eqref{eq:MI} and~\eqref{eq:Mr} together with Lemma~\ref{lemma:d-sets} imply that the number $Z_2$ of ways to choose $W_1(v_\ell), \ldots, W_r(v_\ell)$ so that $v_\ell$ is not useful, and therefore \eqref{eq:MI-not-useful} holds, satisfies
  \[
  \begin{split}
    Z_2 & \le \sum_{I \subseteq [r]} \alpha^{\Ds} \prod_{j \in I} \binom{|V_j|}{\Ds} \cdot \prod_{j \not\in I} \binom{|V_j|}{\Ds} \le 2^r \alpha^{\Ds} \cdot \binom{|V_1| + \ldots + |V_r|}{r \cdot \Ds} \\
    & \le \exp(-6C_1\Ds) \cdot \binom{n}{r\Ds},
  \end{split}
  \]
  where the last inequality follows from~\eqref{eq:lambda-alpha}, provided that $D^* \ge r$, which holds when $n$ is sufficiently large.

  Summarizing the above discussion, similarly as in the proof of Lemma~\ref{lemma:regular-case}, we have
  \begin{equation}
    \label{eq:FF3}
    |\FFI| \le \sum_{t,t^*,x,h} \binom{n}{b} \cdot 2^b \cdot \exp(-6C_1\Ds)^{b/2} \cdot \binom{n}{r\Ds}^b \cdot |\TT_{t'}(t^*,x,h)| \cdot \binom{e(\Pi)}{m-t'-rb\Ds}.
  \end{equation}
  Let $F_3(t,t^*,x,h)$ denote the term in the sum in the right hand side of \eqref{eq:FF3}. Recall that we are counting only graphs $G \in \FFH$ and therefore we may assume that~\eqref{eq:h} holds; a graph $G \in \FFH$ will be counted by $F_3(t,t^*,x,h)$, where $t = e(T(G))$, $t^* = e(T(G)-X(T(G)))$, $x = |X(T(G))|$, and $h = |H(T(G))|$. It follows from Lemma~\ref{lemma:T-count}, see~\eqref{eq:comp-T-count}, that
  \[
  |F_3(t,t^*,x,h)| \le e^{1/\xi} \cdot 2^b \cdot m^{t^* + xD/2} \cdot \exp\left(\frac{4mh}{\xi n}\right) \cdot \exp(-3C_1\Ds b) \cdot \binom{e(\Pi)}{m}.
  \]
  It follows from~\eqref{eq:h} that
  \[
  m^{t^* + xD/2} \le \exp\left(\frac{mh}{n} \cdot \frac{6}{\eps \xi}\right),
  \]
  which, recalling that $b \ge \frac{h}{2r}$ and the definition of $\Ds$ from~\eqref{eq:Ds}, yields
  \[
  \begin{split}
    |F_3(t,t^*,x,h)| & \le e^{1/\xi} \cdot 2^h \cdot \exp\left( \frac{mh}{n} \cdot \left(\frac{6}{\eps \xi} + \frac{4}{\xi} - \frac{3\beta C_1}{2r}\right) \right) \cdot \binom{e(\Pi)}{m} \\
    & \le e^{1/\xi} \cdot 2^h \cdot \exp\left(-\frac{3mh}{n}\right) \cdot \binom{e(\Pi)}{m} \le \exp\left(-\frac{2m}{n}\right) \cdot \binom{e(\Pi)}{m},
  \end{split}
  \]
  where in the first inequality we used~\eqref{eq:C1} and in the last inequality, we used the fact that $h \ge 1$, which follows from \eqref{eq:h} as $h$ is an integer, and that $n$ is sufficiently large. It follows that
  \[
  |\FFI| \le m^2 n^2 \cdot \exp\left(-\frac{2m}{n}\right) \cdot \binom{e(\Pi)}{m} \le \exp\left(-\frac{m}{n}\right) \cdot \binom{e(\Pi)}{m}.
  \]
  This completes the proof in the irregular case.
\end{proof}

\section{The dense case ($m > e(\Pi) - \xi n^2$)}

\label{sec:dense-case}

Recall the definition of $\xi$ given in~(\ref{eq:xi}). In this section, we prove Theorem~\ref{thm:1-statement} in the (easy) case $m > e(\Pi) - \xi n^2$. We begin with a brief sketch of the argument.

\subsection{Outline of the proof}

\label{sec:proof-outline-dense}

Recall the definition of $\TT$ from Section~\ref{sec:setup}. Our proof in the dense case has two main ingredients. In Section~\ref{sec:bounding-Fst-dense}, in Lemma~\ref{lemma:Fst-TTLx-upper}, we give an upper bound on $|\Fst|$, the number of $G \in \Fs$ with $T(G) = T$, in terms of the size of a maximum matching in $T$. In Section~\ref{sec:counting-TT-dense}, in Lemma~\ref{lemma:TTx-upper}, we enumerate graphs $T \in \TT$ with a particular value of this parameter. Combining these two estimates yields the required upper bound on $|\Fs|$.

The bound on the size of $\Fst$ is obtained as follows. First, we note that the family $\Fst$ is empty unless all vertices of $T$ have degree at most $\beta n$, where $\beta$ is some small positive constant. This is because by Claim~\ref{claim:unfriendly}, in every $G \in \Fst$ the neighborhood of such a vertex would contain a large $K_r$-free graph, which, by Lemma~\ref{lemma:k-Turan}, would contradict the assumption that $e(G \cap \Pi) \ge (1-\delta)m > (1-\delta)(e(\Pi) - \xi n^2)$. Second, we observe that in every $G \in \Fst$ the endpoints of every edge of $T$ cannot have many common neighbors in every other (than its own) color class. This is because the set of common neighbors of such an edge induces a $K_{r-1}$-free graph and hence, by Lemma~\ref{lemma:k-Turan} and our assumption that $e(G \cap \Pi) \ge (1-\delta)(e(\Pi) - \xi n^2)$, it cannot be very large. It follows that there are some $i$ and $j$ such that the density of edges between the vertex set of a maximal matching in $T[V_i]$ and $V_j$ is bounded away from $1$. Since by our assumption on $m$, $e(G(\Pi))$ is very close to $e(\Pi)$, this restriction is sufficiently strong to bound the number of choices for $G \cap \Pi$.

\subsection{Setup}

\label{sec:setup-dense}

For every $T \in \TT$, we fix some canonically chosen maximal matching $U(T)$ in $T$ and let $X(T)$ be the set of endpoints of edges in $U(T)$. It follows from the maximality of $U(T)$ that every edge in $T$ has at least one endpoint in $X(T)$. Next, for every $i \in [r]$, we let $U_i(T)$ be the subgraph of $U(T)$ induced by $V_i$, let $X_i(T) = X(T) \cap V_i$, and let $i(T)$ be the smallest index satisfying
\[
|X_{i(T)}(T)| = \max_{i \in [r]} |X_i(T)| \ge \frac{|X(T)|}{r}.
\]
Let
\[
\beta = \frac{1}{40r^3}.
\]
In the argument below, we will use the following inequality, which is a trivial consequence of our choices of $\xi$ and $\delta$:
\begin{equation}
  \label{eq:xi-delta}
  \xi + \delta < \min\left\{ \beta^2, \frac{1}{16r^2} \right\}.
\end{equation}

\subsection{Counting the graphs in $\TT$}

\label{sec:counting-TT-dense}

Let us partition the family $\TT$ according to the size of the set $X(T)$. For an integer $x$, let $\TT(x)$ consist of all $T \in \TT$ that satisfy $|X(T)| = x$. We will use the following trivial upper bound on $|\TT(x)|$.

\begin{lemma}
  \label{lemma:TTx-upper}
  If $n$ is sufficiently large, then for every $x$,
  \[
  |\TT(x)| \le e^{nx}.
  \]
\end{lemma}
\begin{proof}
  We may construct each $T \in \TT(x)$ by selecting $x$ vertices that form the set $X(T)$ and, for each of those vertices, choosing which pairs of vertices intersecting $X(T)$ are edges of $T$. It follows that
  \[
  |\TT(x)| \le \binom{n}{x} 2^{\binom{x}{2} + x(n-x)} \le e^{x(\log n + n\log 2)} \le e^{nx},
  \]
  provided that $n$ is sufficiently large.
\end{proof}

\subsection{Bounding $|\Fst|$ in terms of $|X(T)|$}

\label{sec:bounding-Fst-dense}

We first deal with the case when $T$ contains a vertex with large degree. To this end, let $\TTH$ be the family of all $T \in \TT$ that contain a vertex of degree at least $\beta n$ and let $\TTL = \TT \setminus \TTH$. With our choice of parameters, estimating $|\Fst|$ for $T \in \TTH$ is extremely easy.

\begin{lemma}
  \label{lemma:TTH-empty}
  For every $T \in \TTH$, the family $\Fst$ is empty.
\end{lemma}
\begin{proof}
  Fix a $T \in \TTH$ and let $v$ be an arbitrary vertex with $\deg_T(v) \ge \beta n$. Suppose that $\Fst$ is non-empty and fix an arbitrary $G \in \Fst$. By Claim~\ref{claim:unfriendly} and the definition of $\TTH$, $\deg_G(v,V_i) \ge \beta n$ for every $i \in [r]$. The ($r$-partite) subgraph of $G \cap \Pi$ induced by $N_G(v)$ is $K_r$-free and so by Lemma~\ref{lemma:k-Turan},
  \[
  e(\Pi \setminus G) \ge \min \big\{ \deg_G(v,V_i) \cdot \deg_G(v,V_j) \colon i, j \in [r] \big\} \ge \beta^2n^2.
  \]
  On the other hand, by~\eqref{eq:xi-delta},
  \begin{equation}
    \label{eq:Pi-G-dense}
    e(\Pi \setminus G) = e(\Pi) - e(G) + e(T) \le e(\Pi) - m + \delta m \le \xi n^2 + \delta \binom{n}{2} < \beta^2 n^2,
  \end{equation}
  a contradiction.
\end{proof}

The following lemma is the key step in our proof. It says that for every $G \in \Fst$, there is a fairly large set $X' \subseteq V_{i(T)}$ such that the density of edges between $X'$ and some $V_j$ with $j \neq i(T)$ is at most $3/4$, much lower than the average density of $G \cap \Pi$, which by our assumptions is $1 - O(\xi + \delta)$.

\begin{lemma}
  \label{lemma:hole-in-Pi}
  For every $T \in \TT$ and every $G \in \Fst$, there is a set $X' \subseteq X_{i(T)}$ and a $j \neq i(T)$ such that
  \[
  |X'| \ge \frac{|X_{i(T)}(T)|}{r-1} \qquad \text{and} \qquad e_G(X', V_j) \le \frac{3}{4} |X'| \cdot |V_j|.
  \]
\end{lemma}
\begin{proof}
  Fix some $T \in \TT$ and $G \in \Fst$ and let $i = i(T)$. For every edge $\{u,v\} \in U_{i}(T)$ and every $j \neq i$, let $W_j^{uv}$ be the set of common neighbors of $u$ and $v$ in $V_j$. Suppose first that for every $\{u, v\}$, there is a $j \neq i$ such that $|W_j^{uv}| \le |V_j|/2$, which implies that
  \begin{equation}
    \label{eq:euvVj}
    e_G(\{u,v\}, V_j) \le |V_j| + |W_j^{uv}| \le \frac{3}{2}|V_j|.
  \end{equation}
  Let $j$ be an index for which~\eqref{eq:euvVj} holds for the largest number of edges $\{u,v\} \in U_{i}(T)$ and let $X' \subseteq X_i(T)$ be the set of endpoints of these edges. This set $X'$ clearly satisfies the assertion of this lemma. Suppose now that there is a $\{u, v\} \in U_i(T)$ such that $|W_j^{uv}| > |V_j|/2$ for all $j \neq i$. Since $G$ is $K_{r+1}$-free, the $(r-1)$-partite subgraph of $G \cap \Pi$ induced by the sets $W_j^{uv}$ with $j \neq i$ contains no copy of $K_{r-1}$. In other words, this subgraph of $G \cap \Pi$ is a $K_{r-1}$-free subgraph of the complete $(r-1)$-partite graph with color classes $W_j^{uv}$, where $j \neq i$. It follows from Lemma~\ref{lemma:k-Turan} that the graph $\Pi \setminus G$ contains at least $\min_{j_1 \neq j_2}|W_{j_1}^{uv}||W_{j_2}^{uv}|$ edges. This clearly cannot happen as
  \[
  \min_{j_1 \neq j_2} |W_{j_1}^{uv}||W_{j_2}^{uv}| \ge \min_j \frac{|V_j|^2}{4} \ge \frac{n^2}{16r^2}
  \]
  but, on the other hand, by~\eqref{eq:xi-delta} and~\eqref{eq:Pi-G-dense},
  \[
  e(\Pi \setminus G) \le \xi n^2 + \delta \binom{n}{2} < \frac{n^2}{16r^2},
  \]
  a contradiction.
\end{proof}

Finally, we use Lemma~\ref{lemma:hole-in-Pi} to derive an upper bound on $|\Fst|$ for all $T \in \TTL(x)$.

\begin{lemma}
  \label{lemma:Fst-TTLx-upper}
  If $n$ is sufficiently large, then for every $x$ and every $T \in \TTL(x)$,
  \[
  |\Fst| \le e^{-2nx} \cdot \binom{e(\Pi)}{m}.
  \]
\end{lemma}
\begin{proof}
  By Lemma~\ref{lemma:hole-in-Pi}, we may construct each $G \in \Fst$ by selecting a set $X' \subseteq X_{i(T)}$ of $\lceil x/r^2 \rceil$ vertices, an index $j \neq i(T)$, then choosing some $m'$, where $m' \le \frac{3}{4}|X'||V_j|$, edges between $X'$ and $V_j$, and finally choosing the remaining $m - m' - e(T)$ edges from $\Pi \setminus (X', V_j)$, where $(X', V_j)$ denotes the complete bipartite graph with color classes $X'$ and $V_j$. The reason why $|\Fst|$ is so small is that the number $N_j$ of ways to first choose only at most $\frac{3}{4}|X'||V_j|$ edges between $X'$ and $V_j$ and then the remaining edges from $\Pi \setminus (X', V_j)$ is much smaller than the number of ways to choose $m - e(T)$ edges from $\Pi$. To quantify this, let $t = e(T)$, let $x' = \lceil x/r^2 \rceil$, fix some $j \neq i(T)$, and let $n_j = |V_j|$. Observe that $t \le \beta n x$ due to our assumption that $T \in \TTL(x)$ and therefore,
  \begin{equation}
    \label{eq:hole-in-Pi}
     N_j \le \sum_{m' \le \frac{3}{4}x'n_j} \binom{x'n_j}{m'} \binom{e(\Pi) - x'n_j}{m-m'-t} \le 2^{x'n_j} \cdot \binom{e(\Pi) - x'n_j}{m-\frac{3}{4}x'n_j-\beta nx} \le 2^{xn} \cdot \binom{e(\Pi) - x'n_j}{m-\frac{4}{5}x'n_j},    
  \end{equation}
  where the last two inequalities hold since for each $m'$ with $m' \le \frac{3}{4}x'n_j$,
  \begin{equation}
    \label{eq:ePi-x'n'}
    \frac{e(\Pi) - x'n_j}{2} \le m - \frac{4}{5}x'n_j \le m - \frac{3}{4}x'n_j - \beta n x \le m - m' - t.
  \end{equation}
  To see the first inequality in~\eqref{eq:ePi-x'n'}, recall that $m \ge e(\Pi) - \xi n^2$ and observe that
  \begin{equation}
    \label{eq:xpnp}
    x'n_j \le \frac{|V_{i(T)}|}{2} \cdot |V_j| \le \frac{e(\Pi)}{2}.
  \end{equation}
  To see the second inequality in~\eqref{eq:ePi-x'n'}, recall that $x' \ge x/r^2$, $n_j \ge n/2r$, and therefore $\beta n x \le 2\beta r^3 x'n_j = x'n_j/20$. With the view of further estimating $N_j$, we claim that for every $j \in [r]$,
  \begin{equation}
    \label{eq:hole-in-Pi-2}
    \binom{e(\Pi) - x'n_j}{m-\frac{4}{5}x'n_j} \le e^{-3xn} \cdot \binom{e(\Pi)}{m}.
  \end{equation}
  Assuming that~\eqref{eq:hole-in-Pi-2} holds, \eqref{eq:hole-in-Pi} implies that
  \[
  |\Fst| \le \sum_{j \neq i(T)} \binom{|V_j|}{x'} \cdot N_j \le r \cdot n^{x} \cdot 2^{xn} \cdot e^{-3xn} \cdot \binom{e(\Pi)}{m} \le e^{-2xn} \cdot \binom{e(\Pi)}{m},
  \]
  provided that $n$ is sufficiently large, as required. Thus, it remains to establish~\eqref{eq:hole-in-Pi-2}. To this end, using the trivial identity
  \[
  \binom{a-5}{b-4} = \frac{b(b-1)(b-2)(b-3)(a-b)}{a(a-1)(a-2)(a-3)(a-4)} \cdot \binom{a}{b}
  \]
  and the assumption that $m \ge e(\Pi) - \xi n^2$, we estimate
  \[
  \begin{split}
    \frac{\binom{e(\Pi) - x'n_j}{m - \frac{4}{5}x'n_j}}{\binom{e(\Pi)}{m}} & = \prod_{z = 1}^{x'n_j/5} \frac{\binom{e(\Pi) - 5z}{m-4z}}{\binom{e(\Pi)-5z+5}{m-4z+4}} \le \left(\frac{m^4(e(\Pi)-m)}{(e(\Pi)-x'n_j)^5}\right)^{\frac{x'n_j}{5}} \le \left(\frac{e(\Pi)^4 \cdot \xi n^2}{(e(\Pi)/2)^5}\right)^{\frac{x'n_j}{5}} \\
    & \le (2^5 \cdot 16/3 \cdot \xi)^{\frac{x'n_j}{5}} \le (2^9 \xi / 3)^{\frac{xn}{10r^3}} \le e^{-3xn},
  \end{split}
  \]
  where we used~\eqref{eq:xi}, \eqref{eq:xpnp}, and the fact that $e(\Pi) \ge 3n^2/16$, see~\eqref{eq:ePi-lower}. This completes the proof of the lemma.
\end{proof}

\subsection{Summary}

With Lemmas~\ref{lemma:TTx-upper}, \ref{lemma:TTH-empty},  and \ref{lemma:Fst-TTLx-upper} at our disposal, we can finally deduce an upper bound on $|\Fs|$:
\[
\begin{split}
  |\Fs| & = \sum_{T \in \TT} |\Fst| = \sum_{T \in \TTL} |\Fst| = \sum_{x \ge 1} \sum_{T \in \TTL(x)} |\Fst| \\
  & \le \sum_{x \ge 1} |\TT(x)| \cdot e^{-2nx} \cdot \binom{e(\Pi)}{m} \le \sum_{x \ge 1} e^{-nx} \cdot \binom{e(\Pi)}{m} \le 2e^{-n} \cdot \binom{e(\Pi)}{m}.
\end{split}
\]
This completes the proof of Theorem~\ref{thm:1-statement} in the dense case.

\appendix

\section{Omitted proofs}

\label{sec:omitted-proofs}

\subsection{Tools}

In this section, we prove Lemmas~\ref{lemma:HJI}--\ref{lemma:k-Turan}. In the proofs of the first two of them, we will use the so-called Local LYMB Inequality.

\begin{lemma}[{Local LYMB Inequality}]
  \label{lemma:local-LYMB}
  Let $\HH$ be a $k$-uniform hypergraph on a finite vertex set $V$. The \emph{shadow} of $\HH$ is the $(k-1)$-uniform hypergraph $\sHH$ defined by
  \[
  \sHH = \left\{A \in \binom{V}{k-1} \colon A \subseteq B \text{ for some } B \in \HH \right\}.
  \]
  We have
  \[
  \frac{e(\sHH)}{\binom{|V|}{k-1}} \ge \frac{e(\HH)}{\binom{|V|}{k}}.
  \]
\end{lemma}

\begin{proof}[{Proof of Lemma~\ref{lemma:HJI}}]
  For a set $J \subseteq I$, define
  \[
  \mu(J) = \sum_{i \in J} p^{|B_i|} \qquad \text{and} \qquad \Delta(J) = \sum_{i \sim j} p^{|B_i \cup B_j|},
  \]
  where the second sum is over all ordered pairs $(i,j) \in J^2$ such that $i \neq j$ and $B_i \cap B_j \neq \emptyset$. Let $I_q \subseteq I$ be the $q$-random subset of $I$, that is, the random subset of $I$ where each element of $I$ is included with probability $q$, independently of all other elements, and fix an arbitrary set $J \subseteq I$ that satisfies
  \[
  \mu(J) - \Delta(J)/2 \ge \Ex\big[\mu(I_q) - \Delta(I_q)/2\big] = q\mu - q^2\Delta/2.
  \]
  Let $R'$ be the $p$-random subset of $\Omega$ and let $\BB'$ denote the event that $B_i \nsubseteq R'$ for all $i \in I$; one may think of $R'$ and $\BB'$ as `binomial' analogues of $R$ and $\BB$. By the Janson Inequality (see, e.g., \cite[Theorem~8.1.1]{AlSp}),
  \begin{equation}
    \label{eq:PrBB'}
    \Pr(\BB') \le \Pr(B_i \nsubseteq R' \text{ for all } i \in J) \le \exp\left( -\mu(J) + \Delta(J)/2 \right) \le \exp\left(- q\mu + q^2\Delta/2\right).
  \end{equation}
  It follows from the Local LYMB Inequality (Lemma~\ref{lemma:local-LYMB}) that the function $k \mapsto \Pr(\BB' \mid |R'| = k)$ is decreasing and hence
  \begin{equation}
    \label{eq:PrBB}
    \begin{split}
      \Pr(\BB') & = \sum_{k=0}^n \Pr(\BB' \mid |R'| = k) \cdot \Pr(|R'| = k) \ge \Pr(\BB' \mid |R'| = m) \cdot \Pr(|R'| \le m) \\
      & = \Pr(\BB) \cdot \Pr(|R'| \le m) \ge \Pr(\BB)/2,
    \end{split}
  \end{equation}
  where the last inequality follows from the well-known fact that if $np$ is an integer, then it is the median of the binomial distribution $\Bin(n,p)$. Inequalities~\eqref{eq:PrBB'} and~\eqref{eq:PrBB} readily imply the claimed bound on $\Pr(\BB)$.
\end{proof}

\begin{proof}[{Proof of Lemma~\ref{lemma:HFKG}}]
  Set $p = (1+\eta)m/n$ and note that $p \le 1$ by our assumptions on $m$ and $\eta$. Let $R'$ be the $p$-random subset of $\Omega$ and let $\BB'$ denote the event that $B_i \nsubseteq R'$ for all $i \in I$. By the FKG Inequality (see, e.g., \cite[Chapter~6]{AlSp}),
  \begin{equation}
    \label{eq:PrBB'-FKG}
    \Pr(\BB') \ge \prod_{i \in I} \left(1 - p^{|B_i|}\right).
  \end{equation}
  It follows from the Local LYMB Inequality (Lemma~\ref{lemma:local-LYMB}) that the function $k \mapsto \Pr(\BB' \mid |R'| = k)$ is decreasing and hence
  \begin{equation}
    \label{eq:PrBB-FKG}
    \begin{split}
      \Pr(\BB') & = \sum_{k=0}^n \Pr(\BB' \mid |R'| = k) \cdot \Pr(|R'| = k) \le \Pr(\BB' \mid |R'| = m) + \Pr(|R'| < m) \\
      & = \Pr(\BB) + \Pr(|R'| < m).
    \end{split}
  \end{equation}
  The claimed bound now easily follows from \eqref{eq:PrBB'-FKG} and \eqref{eq:PrBB-FKG} as by Chernoff's Inequality (see, e.g., \cite[Appendix~A]{AlSp}),
  \[
  \Pr(|R'| < m) \le \exp\left(-\frac{\eta^2m^2}{2(1+\eta)m}\right) \le \exp\left(-\frac{\eta^2m}{4}\right). \qedhere
  \]
\end{proof}

\begin{proof}[{Proof of Lemma~\ref{lemma:k-Turan}}]
  Let us denote the graph $K(n_1, \ldots, n_r)$ by $G$. Let $V_1, \ldots, V_r$ be the color classes of $G$ with $n_1, \ldots, n_r$ elements, respectively. Clearly, deleting all edges between $V_1$ and $V_2$ removes all copies of $K_r$ from $G$ and hence $\ex(G, K_r) \ge e(G) - n_1n_2$. We prove the converse inequality by induction on $r$. The statement is trivial if $r = 2$, so let us assume that $r \ge 3$. Let $H$ be a $K_r$-free subgraph of $G$. Let $\Delta = \max_{v \in V_r} \deg_H(v)$ and fix an arbitrary $v \in V_r$ with $\deg_H(v) = \Delta$. For each $i \in [r-1]$, let $d_i = \deg_H(v,V_i)$. Since the subgraph of $G$ induced by $N_H(v)$ is a $K_{r-1}$-free subgraph of $K(d_1, \ldots, d_{r-1})$, it follows from our inductive assumption that
  \[
  e(G) - e(H) \ge e_{G \setminus H}(V_1 \cup \ldots \cup V_{r-1}, V_r) + e_{G \setminus H}(N_H(v)) \ge n_r \cdot \left(\sum_{k=1}^{r-1}n_k - \Delta\right) + \min_{i < j} d_id_j.
  \]
  Let $\{i,j\}$, where $i < j$, be a pair of indices for which $d_i d_j$ attains its minimum value. Since $d_1 + \ldots + d_{r-1} = \Delta$ by our choice of $v$, then
  \[
  \begin{split}
    e(G) - e(H) & \ge n_r \cdot \sum_{k=1}^{r-1} (n_k - d_k) + d_id_j \ge n_r \cdot [(n_i - d_i) + (n_j - d_j)] + d_id_j \\
    & \ge n_j \cdot [(n_i - d_i) + (n_j - d_j)] + d_id_j = n_in_j + (n_j-d_j)(n_j-d_i) \ge n_in_j \ge n_1n_2,
  \end{split}
  \]
  as claimed. It is not hard to verify that $e(G) - e(H) > n_1n_2$ unless $H = G \setminus (V_i,V_j)$, where $i$ and $j$ are such that $n_in_j = n_1n_2$.
\end{proof}

\subsection{Main tool}

Lemma~\ref{lemma:d-sets} is a straightforward consequence of the following somewhat technical (but tailored to facilitate an inductive proof) statement.

\begin{lemma}
  \label{lemma:d-sets-tech}
  Let $\alpha, \lambda \in (0,1)$, let $V_1, \ldots, V_k$ be finite sets, and let $d$ be an integer satisfying $2 \le d \le \min\{ |V_1|, \ldots, |V_k| \}$. Suppose that $\HH \subseteq V_1 \times \ldots \times V_k$ satisfies
  \[
  |\HH| \le (\alpha \lambda)^k \prod_{i = 1}^k |V_i|.
  \]
  Then for all but at most
  \[
  \big(d^k - 1\big) \big(2\alpha^\lambda\big)^d \prod_{i=1}^k \binom{|V_i|}{d}
  \]
  choices of $W_1 \in \binom{V_1}{d}, \ldots, W_k \in \binom{V_k}{d}$, we have
  \begin{equation}
    \label{eq:prod-W-HH}
    |\HH \cap (W_1 \times \ldots \times W_k)| \le k \lambda d^k.
  \end{equation}
\end{lemma}

We will deduce this statement from the following one-sided version of Hoeffding's inequality~\cite{Ho63} for the hypergeometric distribution.

\begin{lemma}
  \label{lemma:lrg-dev}
  Let $d$ and $n$ be integers and let $X$ denote the uniformly chosen random $d$-subset of $[n]$. Then for every $\alpha, \lambda \in (0,1)$,
  \begin{equation}
    \label{eq:lrg-dev}
    \Pr\left( \big| X \cap [\alpha \lambda n] \big| \ge \lambda d \right) \le \big(2 \alpha^\lambda\big)^d.
  \end{equation}
\end{lemma}
\begin{proof}
  Denote the left-hand side of~\eqref{eq:lrg-dev} by $P$. It follows from Hoeffding's inequality~\cite[Theorem~1]{Ho63} that\footnote{Even though~\cite[Theorem~1]{Ho63} applies to sums of independent random variables, the bound obtained there remains valid for the hypergeometric distribution, see the discussion in \cite[Section~6]{Ho63}.}
  \[
  P^{1/d} \le \left(\frac{\alpha \lambda}{\lambda}\right)^\lambda \left(\frac{1 - \alpha \lambda}{1- \lambda}\right)^{1 - \lambda} \le \alpha^\lambda \left(\frac{1}{1-\lambda}\right)^{1-\lambda} \le \alpha^\lambda e^{1/e} \le 2\alpha^\lambda,
  \]
  where the third inequality follows from the fact that if $a > 0$, then the function $x \mapsto (a/x)^x$ attains its maximum value at $x = a/e$.
\end{proof}

\begin{proof}[{Proof of Lemma~\ref{lemma:d-sets-tech}}]
  We prove the statement by induction on $k$. The induction base ($k = 1$) follows directly from Lemma~\ref{lemma:lrg-dev} and the fact that $d - 1 \ge 1$. For the induction step, assume that $k \ge 2$. For every $v \in V_k$, let
  \[
  \HH_v = \{(v_1, \ldots, v_{k-1}) \in V_1 \times \ldots \times V_{k-1} \colon (v_1, \ldots, v_{k-1}, v) \in \HH \}.
  \]
  One may think of $\HH_v$ as the link hypergraph of the vertex $v$. Roughly speaking, we shall argue as follows. If $W_1 \times \ldots \times W_k$ contains many edges of $\HH$, then either
  \begin{enumerate}[(i)]
  \item
    \label{item:d-sets-tech-1}
    the set $W_k$ contains many vertices whose degree in $\HH$ is high or
  \item
    \label{item:d-sets-tech-2}
    for some $v \in W_k$ whose degree in $\HH$ is low, the set $W_1 \times \ldots \times W_{k-1}$ has a large intersection with the link hypergraph $\HH_v$.
  \end{enumerate}
  The desired bound will then follow by applying Lemma~\ref{lemma:lrg-dev} (when~(\ref{item:d-sets-tech-1}) holds) and the induction hypothesis (when~(\ref{item:d-sets-tech-2}) holds). The details now follow.

  Let
  \[
  V_k' = \Big\{ v \in V_k \colon |\HH_v| > (\alpha \lambda)^{k-1} \prod_{i=1}^{k-1} |V_i| \Big\}.
  \]
  Intuitively, $V_k'$ is the set of all $v \in V_k$ whose degree in $\HH$ exceeds the assumed upper bound on the average degree of $V_k$ in $\HH$ by a factor of more than $(\alpha \lambda)^{-1}$. Note that our assumption on $\HH$ implies that $|V_k'| < \alpha \lambda |V_k|$. Furthermore, let
  \[
  \WW_k = \Big\{ W \in \binom{V_k}{d} \colon |W \cap V_k'| \ge \lambda d \Big\}
  \]
  and for each $W \in \binom{V_k}{d} \setminus \WW_k$, define $\WW_W \subseteq \binom{V_1}{d} \times \ldots \times \binom{V_{k-1}}{d}$ by
  \[
  \WW_W = \big\{ (W_1, \ldots, W_{k-1}) \colon |\HH_v \cap W_1 \times \ldots \times W_{k-1}| \ge (k-1)\lambda d^{k-1} \text{ for some $v \in W \setminus V_k'$} \big\}.
  \]
  Note that if $W_1 \in \binom{V_1}{d}, \ldots, W_k \in \binom{V_k}{d}$ are such that $W_k \not\in \WW_k$ and $(W_1, \ldots, W_{k-1}) \not\in \WW_{W_k}$, then
  \[
  |\HH \cap W_1 \times \ldots \times W_k| \le \lambda d \cdot d^{k-1} + d \cdot (k-1)\lambda d^{k-1} \le k\lambda d^k.
  \]
  Hence, the number $B$ of $k$-tuples $(W_1, \ldots, W_k)$ for which~\eqref{eq:prod-W-HH} does not hold satisfies
  \begin{equation}
    \label{eq:bad-tuples}
    B \le |\WW_k| \cdot \prod_{i=1}^{k-1} \binom{|V_i|}{d} + \sum_{W \not \in \WW_k} |\WW_W|.    
  \end{equation}
  By Lemma~\ref{lemma:lrg-dev},
  \begin{equation}
    \label{eq:WWk}
    |\WW_k| \le \big(2 \alpha^\lambda)^d \binom{|V_k|}{d}    
  \end{equation}
  and, since
  \[
  |\HH_v| \le (\alpha \lambda)^{k-1} \prod_{i=1}^{k-1} |V_i|
  \]
  for every $v \not\in V_k'$, then by our inductive assumption, for every $W \not\in \WW_k$,
  \begin{equation}
    \label{eq:WWW}
    |\WW_W| \le \sum_{v \in W \setminus V_k'} \big(d^{k-1}-1\big)\big(2 \alpha^\lambda\big)^d \prod_{i=1}^{k-1}\binom{|V_i|}{d} \le \big(d^k-d\big)\big(2\alpha^\lambda\big)^d \prod_{i=1}^{k-1}\binom{|V_i|}{d}.
  \end{equation}
  Putting~\eqref{eq:bad-tuples}, \eqref{eq:WWk}, and~\eqref{eq:WWW} together yields
  \[
  B \le \big(1 + d^k - d\big)\big(2\alpha^\lambda)^d \prod_{i=1}^k \binom{|V_i|}{d} \le \big(d^k-1\big)\big(2\alpha^\lambda\big)^d \prod_{i=1}^k\binom{|V_i|}{d},
  \]
  as claimed.
\end{proof}

\subsection{Non-uniquely colorable and unbalanced graphs}

Finally, we present the proofs of Propositions~\ref{prop:unbalanced-graphs} and~\ref{prop:UP}.

\begin{proof}[{Proof of Proposition~\ref{prop:unbalanced-graphs}}]
  Fix an arbitrary partition $\Pi$ that does not satisfy~\eqref{eq:Part-gamma} and observe that (see~\eqref{eq:ePi-gamma-upper})
  \[
  e(\Pi) \le \exnr - \frac{\gamma^2 n^2}{2}.
  \]
  Consequently, by Lemma~\ref{lemma:acbc}, we have that
  \[
  |\GP| = \binom{e(\Pi)}{m} \le \left(\frac{\exnr - \frac{\gamma^2 n^2}{2}}{\exnr}\right)^{m} \cdot \binom{\exnr}{m} \le e^{-\gamma^2 m} \cdot \binom{\exnr}{m}.
  \]
  To complete the proof, we just observe that there are at most $r^n$ different $r$-colorings and that $r^n \le e^{\gamma^2 m/2}$ if $m \ge cn$ for a sufficiently large constant $c$.
\end{proof}

\begin{proof}[{Proof of Proposition~\ref{prop:UP}}]
  Fix some $\Pi \in \Part(\frac{1}{2r})$ and $\Pi' \in \Part \setminus \{\Pi\}$. Suppose that $\Pi = \{V_1, \ldots, V_r\}$ and $\Pi' = \{V_1', \ldots, V_r'\}$ and for all $i, j \in [r]$, let $V_{i,j} = V_i \cap V_j'$. We will say that the vertices in $V_{i,j}$ are \emph{moved} from $V_i$ to $V_j'$. For every $i \in [r]$, define $L_i$ and $S_i$ as the largest and the second largest subclasses of $V_i$, respectively. Note that $|V_i| \ge \frac{n}{2r}$ implies that $|L_i| \ge \frac{n}{2r^2}$. Set $s = \max_{j \in [r]} |S_j|$ and let $S = S_j$ for the smallest $j$ for which the maximum in the definition of $s$ is achieved. Note that $1 \le s \le n/2$, as $s = 0$ would imply that $(V_1, \ldots, V_r)$ is a permutation of $(V_1', \ldots, V_r')$, and therefore $\Pi = \Pi'$.

  Observe that, by the pigeonhole principle, either some pair $\{L_i, L_j\}$ of largest subclasses, or some largest subclass $L_i$ and $S$, where $S \nsubseteq V_i$, are moved to the same vertex class $V_k'$. Since $V_k'$ is an independent set in every $G \in \GPp$, it follows that every $G \in \GP \cap \GPp$ has no edges between these sets $L_i$ and $L_j$ or $L_i$ and $S$. Since,
  \[
  \min\{ |L_i| \cdot |L_j|, |L_i| \cdot |S| \} \ge \cdot \min\left\{ \left(\frac{n}{2r^2}\right)^2, \frac{n}{2r^2} \cdot s \right\} \ge \frac{sn}{2r^4},
  \]
  it follows from Lemma~\ref{lemma:acbc} and the inequality $e(\Pi) \le n^2/2$ that, if $m \ge r^6 (a+3) n \log n$,
  \begin{equation}
    \label{eq:GP-GPp}
    |\GP \cap \GPp| \le \binom{e(\Pi) - \frac{sn}{2r^4}}{m} \le \left(1 - \frac{s}{nr^4}\right)^m \binom{e(\Pi)}{m} \le n^{-(a+3)sr^2} \cdot |\GP|.
  \end{equation}

  Finally, observe that given a $\Pi$, we can describe any $\Pi' \neq \Pi $ by first picking the partitions $(V_{i,j})_{j \in [r]}$ for every $i$ and then setting $V_j' = \bigcup_{i \in [r]} V_{i,j}$. A moment's thought reveals that for every $s$, there are at most $n^{r^2} \cdot n^{sr^2}$ ways to choose  all $V_{i,j}$ so that $\max_{i \in [r]} |S_i| = s$. Therefore, by~\eqref{eq:GP-GPp},
  \[
  |\GP \setminus \UP| \le \sum_{s \ge 1} \left( n^{(s+1)r^2} \cdot n^{-s(a+3)r^2} \right) \cdot |\GP| \le n^{-a} \cdot |\GP|,
  \]
  which completes the proof.
\end{proof}

\bigskip
\noindent
\textbf{Acknowledgments.}
We would like to thank the anonymous referee for a careful reading of the paper and several valuable comments and suggestions. Lutz Warnke would like to thank Angelika Steger for numerous insightful discussions on the topic of this paper.

\bibliographystyle{amsplain}
\bibliography{Kr-free}

\providecommand{\bysame}{\leavevmode\hbox to3em{\hrulefill}\thinspace}
\providecommand{\MR}{\relax\ifhmode\unskip\space\fi MR }
\providecommand{\MRhref}[2]{%
  \href{http://www.ams.org/mathscinet-getitem?mr=#1}{#2}
}
\providecommand{\href}[2]{#2}
\begin{thebibliography}{10}

\bibitem{AcFr99}
D.~Achlioptas and E.~Friedgut, \emph{A sharp threshold for {$k$}-colorability},
  Random Structures Algorithms \textbf{14} (1999), 63--70.

\bibitem{AcNa}
D.~Achlioptas and A.~Naor, \emph{The two possible values of the chromatic
  number of a random graph}, Ann. of Math. (2) \textbf{162} (2005), 1335--1351.

\bibitem{AlBaBoMo11}
N.~Alon, J.~Balogh, B.~Bollob{\'a}s, and R.~Morris, \emph{The structure of
  almost all graphs in a hereditary property}, J. Combin. Theory Ser. B
  \textbf{101} (2011), 85--110.

\bibitem{AlSp}
N.~Alon and J.~H. Spencer, \emph{The probabilistic method}, third ed.,
  Wiley-Interscience Series in Discrete Math. and Optimization, John Wiley \&
  Sons Inc., Hoboken, NJ, 2008.

\bibitem{BaSiSp90}
L.~Babai, M.~Simonovits, and J.~Spencer, \emph{Extremal subgraphs of random
  graphs}, J. Graph Theory \textbf{14} (1990), 599--622.

\bibitem{BaBoSi04}
J.~Balogh, B.~Bollob{\'a}s, and M.~Simonovits, \emph{The number of graphs
  without forbidden subgraphs}, J. Combin. Theory Ser. B \textbf{91} (2004),
  1--24.

\bibitem{BaBoSi09}
\bysame, \emph{The typical structure of graphs without given excluded
  subgraphs}, Random Structures Algorithms \textbf{34} (2009), 305--318.

\bibitem{BaBoSi11}
\bysame, \emph{The fine structure of octahedron-free graphs}, J. Combin. Theory
  Ser. B \textbf{101} (2011), 67--84.

\bibitem{BaBu11}
J.~Balogh and J.~Butterfield, \emph{Excluding induced subgraphs: critical
  graphs}, Random Structures Algorithms \textbf{38} (2011), 100--120.

\bibitem{BaMoSa12}
J.~Balogh, R.~Morris, and W.~Samotij, \emph{Independent sets in hypergraphs},
  to appear in J. Amer. Math. Soc.

\bibitem{BaMu11}
J.~Balogh and D.~Mubayi, \emph{Almost all triple systems with independent
  neighborhoods are semi-bipartite}, J. Combin. Theory Ser. A \textbf{118}
  (2011), 1494--1518.

\bibitem{BaMu12}
\bysame, \emph{Almost all triangle-free triple systems are tripartite},
  Combinatorica \textbf{32} (2012), 143--169.

\bibitem{Bo}
B.~Bollob{\'a}s, \emph{Random graphs}, second ed., Cambridge Studies in
  Advanced Mathematics, vol.~73, Cambridge University Press, Cambridge, 2001.

\bibitem{BrPaSt12}
G.~Brightwell, K.~Panagiotou, and A.~Steger, \emph{Extremal subgraphs of random
  graphs}, Random Structures Algorithms \textbf{41} (2012), 147--178.

\bibitem{CoVi13}
A.~Coja-Oghlan and D.~Vilenchik, \emph{Chasing the $k$-colorability threshold},
  arXiv:1304.1063 [cs.DM].

\bibitem{CoGo}
D.~Conlon and T.~Gowers, \emph{Combinatorial theorems in sparse random sets},
  arXiv:1011.4310v1 [math.CO].

\bibitem{CoGoSaSc}
D.~Conlon, W.~T. Gowers, W.~Samotij, and M.~Schacht, \emph{On the {K\L R}
  conjecture in random graphs}, to appear in Israel J. Math.

\bibitem{DMKa}
B.~DeMarco and J.~Kahn, \emph{Mantel's theorem for random graphs}, to appear in
  Random Structures Algorithms.

\bibitem{DMKa-Kr}
\bysame, \emph{Tur{\'a}n's theorem for random graphs}, in preparation.

\bibitem{ErKlRo}
P.~Erd{\H{o}}s, D.~J. Kleitman, and B.~L. Rothschild, \emph{Asymptotic
  enumeration of {$K_n$}-free graphs}, Colloquio {I}nternazionale sulle
  {T}eorie {C}ombinatorie ({R}ome, 1973), {T}omo {II}, Accad. Naz. Lincei,
  Rome, 1976, pp.~19--27. Atti dei Convegni Lincei, No. 17.

\bibitem{ErRe60}
P.~Erd{\H{o}}s and A.~R{\'e}nyi, \emph{On the evolution of random graphs},
  Magyar Tud. Akad. Mat. Kutat\'o Int. K\"ozl. \textbf{5} (1960), 17--61.

\bibitem{ErSi66}
P.~Erd{\H{o}}s and M.~Simonovits, \emph{A limit theorem in graph theory},
  Studia Sci. Math. Hungar. \textbf{1} (1966), 51--57.

\bibitem{FrRo86}
P.~Frankl and V.~R{\"o}dl, \emph{Large triangle-free subgraphs in graphs
  without {$K_4$}}, Graphs Combin. \textbf{2} (1986), 135--144.

\bibitem{Fr99}
E.~Friedgut, \emph{Sharp thresholds of graph properties, and the {$k$}-sat
  problem}, J. Amer. Math. Soc. \textbf{12} (1999), 1017--1054, With an
  appendix by Jean Bourgain.

\bibitem{FrRoSc10}
E.~Friedgut, V.~R{\"o}dl, and M.~Schacht, \emph{Ramsey properties of random
  discrete structures}, Random Structures Algorithms \textbf{37} (2010),
  407--436.

\bibitem{Ho63}
W.~Hoeffding, \emph{Probability inequalities for sums of bounded random
  variables}, J. Amer. Statist. Assoc. \textbf{58} (1963), 13--30.

\bibitem{HuPrSt93}
C.~Hundack, H.~J. Pr{\"o}mel, and A.~Steger, \emph{Extremal graph problems for
  graphs with a color-critical vertex}, Combin. Probab. Comput. \textbf{2}
  (1993), 465--477.

\bibitem{JaLuRu}
S.~Janson, T.~{\L}uczak, and A.~Rucinski, \emph{Random graphs},
  Wiley-Interscience Series in Discrete Mathematics and Optimization,
  Wiley-Interscience, New York, 2000.

\bibitem{KoLuRo96}
Y.~Kohayakawa, T.~{\L}uczak, and V.~R{\"o}dl, \emph{Arithmetic progressions of
  length three in subsets of a random set}, Acta Arith. \textbf{75} (1996),
  133--163.

\bibitem{KoLuRo97}
\bysame, \emph{On {$K\sp 4$}-free subgraphs of random graphs}, Combinatorica
  \textbf{17} (1997), 173--213.

\bibitem{KoPrRo87}
P.~Kolaitis, H.~J. Pr{\"o}mel, and B.~Rothschild, \emph{{$K_{l+1}$}-free
  graphs: asymptotic structure and a $0-1$ law}, Trans. Amer. Math. Soc.
  \textbf{303} (1987), 637--671.

\bibitem{Lu00}
T.~{\L}uczak, \emph{On triangle-free random graphs}, Random Structures
  Algorithms \textbf{16} (2000), 260--276.

\bibitem{OsPrTa03}
D.~Osthus, H.~J. Pr{\"o}mel, and A.~Taraz, \emph{For which densities are random
  triangle-free graphs almost surely bipartite?}, Combinatorica \textbf{23}
  (2003), 105--150.

\bibitem{PeSc09}
Y.~Person and M.~Schacht, \emph{Almost all hypergraphs without {F}ano planes
  are bipartite}, Proceedings of the {T}wentieth {A}nnual {ACM}-{SIAM}
  {S}ymposium on {D}iscrete {A}lgorithms (Philadelphia, PA), SIAM, 2009,
  pp.~217--226.

\bibitem{PrSt92}
H.~J. Pr{\"o}mel and A.~Steger, \emph{The asymptotic number of graphs not
  containing a fixed color-critical subgraph}, Combinatorica \textbf{12}
  (1992), 463--473.

\bibitem{PrSt95}
\bysame, \emph{Random {$l$}-colorable graphs}, Random Structures Algorithms
  \textbf{6} (1995), 21--37.

\bibitem{PrSt96}
\bysame, \emph{On the asymptotic structure of sparse triangle free graphs}, J.
  Graph Theory (1996), no.~2, 137--151.

\bibitem{RoRu95}
V.~R{\"o}dl and A.~Ruci{\'n}ski, \emph{Threshold functions for {R}amsey
  properties}, J. Amer. Math. Soc. \textbf{8} (1995), 917--942.

\bibitem{RoRu97}
\bysame, \emph{Rado partition theorem for random subsets of integers}, Proc.
  London Math. Soc. (3) \textbf{74} (1997), 481--502.

\bibitem{RoSc13}
V.~R{\"o}dl and M.~Schacht, \emph{Extremal results in random graphs}, Erd\"os
  centennial, Bolyai Soc. Math. Stud., vol.~25, J\'anos Bolyai Math. Soc.,
  Budapest, 2013, pp.~535--583.

\bibitem{Sa14}
W.~Samotij, \emph{Stability results for random discrete structures}, Random
  Structures Algorithms \textbf{44} (2014), 269--289.

\bibitem{SaTh}
D.~Saxton and A.~Thomason, \emph{Hypergraph containers}, arXiv:1204.6595v2
  [math.CO].

\bibitem{Sc}
M.~Schacht, \emph{Extremal results for random discrete structures}, submitted.

\bibitem{Si68}
M.~Simonovits, \emph{A method for solving extremal problems in graph theory,
  stability problems}, Theory of {G}raphs ({P}roc. {C}olloq., {T}ihany, 1966),
  Academic Press, New York, 1968, pp.~279--319.

\bibitem{St05}
A.~Steger, \emph{On the evolution of triangle-free graphs}, Combin. Probab.
  Comput. \textbf{14} (2005), 211--224.

\bibitem{Tu41}
P.~Tur{\'a}n, \emph{Eine {E}xtremalaufgabe aus der {G}raphentheorie}, Mat. Fiz.
  Lapok \textbf{48} (1941), 436--452.

\bibitem{Wa09}
L.~Warnke, \emph{On the evolution of ${K_4}$-free graphs}, Master's thesis, ETH
  Z{\"u}rich, 2009.

\end{thebibliography}

\end{document}